\documentclass[a4paper, reqno, 12pt]{amsart}

\usepackage[usenames,dvipsnames]{color}
\usepackage{amsthm,amsfonts,amssymb,amsmath,amsxtra,tikz-cd}

\usepackage{float}
\restylefloat{table}

\usepackage{tikz-cd}
\usepackage{tikz}
\usetikzlibrary{arrows}
\usepackage{graphicx}
\usepackage{todonotes,cancel}

\usepackage[all]{xy}
\SelectTips{cm}{}
\usepackage{xr-hyper}
\usepackage[colorlinks=
   citecolor=Black,
   linkcolor=Red,
   urlcolor=Blue]{hyperref}
\usepackage{verbatim}

\usepackage[margin=1.25in]{geometry}
\usepackage{mathrsfs}

\RequirePackage{xspace}
\RequirePackage{etoolbox}
\RequirePackage{varwidth}
\RequirePackage{enumitem}
\RequirePackage{tensor}
\RequirePackage{mathtools}
\RequirePackage{longtable}
\RequirePackage{multirow}

\def\ge{\geqslant}
\def\le{\leqslant}
\def\a{\alpha}
\def\b{\beta}

\def\D{\Delta}
\def\L{\Lambda}

\def\s{\sigma}

\def\l{\lambda}

\def\i{^{-1}}

\def\<{\langle}
\def\>{\rangle}

\newcommand{{\BG}}{\ensuremath{\mathbb {G}}\xspace}

\newcommand{{\BK}}{\ensuremath{\mathbb {K}}\xspace}

\newcommand{\BN}{\ensuremath{\mathbb {N}}\xspace}

\newcommand{\BZ}{\ensuremath{\mathbb {Z}}\xspace}

\newcommand{\CB}{\ensuremath{\mathcal {B}}\xspace}

\newcommand{\CD}{\ensuremath{\mathcal {D}}\xspace}

\newcommand{\CP}{\ensuremath{\mathcal {P}}\xspace}

\newcommand{\CY}{\ensuremath{\mathcal {Y}}\xspace}

\def\tw{\tilde w}
\def\tW{\tilde W}

\def\kk{\mathbf k}

\newtheorem{theorem}{Theorem}
\newtheorem{prop}[theorem]{Proposition}
\newtheorem{proposition}[theorem]{Proposition}
\newtheorem{lem}[theorem]{Lemma}
\newtheorem{lemma}[theorem]{Lemma}

\newtheorem{cor}[theorem]{Corollary}

\theoremstyle{definition}

\newtheorem{remark}[theorem]{Remark}

\numberwithin{equation}{section}
\numberwithin{theorem}{section}

\setitemize[0]{leftmargin=*,itemsep=\the\smallskipamount}
\setenumerate[0]{leftmargin=*,itemsep=\the\smallskipamount}

\renewcommand{\to}{%
   \ifbool{@display}{\longrightarrow}{\rightarrow}%
   }
\let\shortmapsto\mapsto
\renewcommand{\mapsto}{%
   \ifbool{@display}{\longmapsto}{\shortmapsto}%
   }
\newlength{\olen}
\newlength{\ulen}
\newlength{\xlen}
\newcommand{\xra}[2][]{%
   \ifbool{@display}%
      {\settowidth{\olen}{$\overset{#2}{\longrightarrow}$}%
       \settowidth{\ulen}{$\underset{#1}{\longrightarrow}$}%
       \settowidth{\xlen}{$\xrightarrow[#1]{#2}$}%
       \ifdimgreater{\olen}{\xlen}%
          {\underset{#1}{\overset{#2}{\longrightarrow}}}%
          {\ifdimgreater{\ulen}{\xlen}%
             {\underset{#1}{\overset{#2}{\longrightarrow}}}
             {\xrightarrow[#1]{#2}}}}%
      {\xrightarrow[#1]{#2}}
   }
\makeatother
\newcommand{\xyra}[2][]{%
   \settowidth{\xlen}{$\xrightarrow[#1]{#2}$}%
   \ifbool{@display}%
      {\settowidth{\olen}{$\overset{#2}{\longrightarrow}$}%
       \settowidth{\ulen}{$\underset{#1}{\longrightarrow}$}%
       \ifdimgreater{\olen}{\xlen}%
          {\mathrel{\xymatrix@M=.12ex@C=3.2ex{\ar[r]^-{#2}_-{#1} &}}}%
          {\ifdimgreater{\ulen}{\xlen}%
             {\mathrel{\xymatrix@M=.12ex@C=3.2ex{\ar[r]^-{#2}_-{#1} &}}}
             {\mathrel{\xymatrix@M=.12ex@C=\the\xlen{\ar[r]^-{#2}_-{#1} &}}}}}%
      {\mathrel{\xymatrix@M=.12ex@C=\the\xlen{\ar[r]^-{#2}_-{#1} &}}}%
   }
\makeatletter
\newcommand{\xla}[2][]{%
   \ifbool{@display}%
      {\settowidth{\olen}{$\overset{#2}{\longleftarrow}$}%
       \settowidth{\ulen}{$\underset{#1}{\longleftarrow}$}%
       \settowidth{\xlen}{$\xleftarrow[#1]{#2}$}%
       \ifdimgreater{\olen}{\xlen}%
          {\underset{#1}{\overset{#2}{\longleftarrow}}}%
          {\ifdimgreater{\ulen}{\xlen}%
             {\underset{#1}{\overset{#2}{\longleftarrow}}}
             {\xleftarrow[#1]{#2}}}}%
      {\xleftarrow[#1]{#2}}
   }
\newcommand{\isoarrow}{%
   \ifbool{@display}{\overset{\sim}{\longrightarrow}}{\xrightarrow\sim}%
   }
   
  \newcommand{\leJ}{ {\, {}^J \!\! \le \,} }
  
  \newcommand{\leI}{ {\, {}^{I^\flat} \!\! \le \,} }

\begin{document}

\title[]{A Birkhoff-Bruhat Atlas for partial flag varieties}

\author[Huanchen Bao]{Huanchen Bao}
\address{Department of Mathematics, National University of Singapore, Singapore.}
\email{huanchen@nus.edu.sg}

\author[Xuhua He]{Xuhua He}
\address{The Institute of Mathematical Sciences and Department of Mathematics, The Chinese University of Hong Kong, Shatin, N.T., Hong Kong SAR, China}
\email{xuhuahe@math.cuhk.edu.hk}
\thanks{}


\maketitle


\begin{abstract}
A partial flag variety $\CP_K$ of a Kac-Moody group $G$ has a natural stratification into projected Richardson varieties. When $G$ is a connected reductive group, a Bruhat atlas for $\CP_K$ was constructed in \cite{HKL}: $\CP_K$ is locally modeled with Schubert varieties in some Kac-Moody flag variety as stratified spaces. The existence of Bruaht atlases implies some nice combinatorial and geometric properties on the partial flag varieties and the decomposition into projected Richardson varieties. 

A Bruhat atlas does not exist for partial flag varieties of an arbitrary Kac-Moody group due to combinatorial and geometric reasons. To overcome obstructions, we introduce the notion of Birkhoff-Bruhat atlas. Instead of the Schubert varieties used in a Bruhat atlas, we use the $J$-Schubert varieties for a Birkhoff-Bruhat atlas. The notion of the $J$-Schubert varieties interpolates Birkhoff decomposition and Bruhat decomposition of the full flag variety (of a larger Kac-Moody group). The main result of this paper is the construction of a Birkhoff-Bruhat atlas for any partial flag variety $\CP_K$ of a Kac-Moody group. We also construct a combinatorial atlas for the index set $Q_K$ of the projected Richardson varieties in $\CP_K$. As a consequence, we show that $Q_K$ has some nice combinatorial properties. This gives a new proof and generalizes the work of Williams \cite{LW} in the case where the group $G$ is a connected reductive group. 
\end{abstract}

\section{Introduction}

\subsection{The flag variety and its decomposition into Richardson varieties}

Let $G$ be a connected reductive group and $\CB$ be the full flag variety of $G$. An open Richardson variety is the intersection of a Bruhat cell with an opposite Bruhat cell. We then have the decomposition of $\CB$ into the disjoint union of the open Richardson varieties. This decomposition has many remarkable properties, including:
\begin{enumerate}
    \item each stratum is smooth; 
    \item the closure of each stratum is a union of other strata; 
    \item the closure of each stratum is normal, Cohen-Macaulay, with rational singularities;
    \item over positive characteristic, there exists a Frobenius splitting on $\CB$ which compatibly splits all the strata;
    \item over complex numbers, there exists a Poisson structure on $\CB$ for which the $T$-leaves are exactly the strata;
    \item over real numbers, the intersection of the totally nonnegative flag variety $X_{\ge 0}$ with each stratum gives a cellular decomposition of $X_{\ge 0}$; 
    \item the poset of the strata is thin and EL-shellable.
\end{enumerate}

Many of these remarkable properties remain valid for the full flag variety of any Kac-Moody group. 

\subsection{Partial flag varieties} Now we consider the partial flag variety $\CP_K=G/P_K$ of a Kac-Moody group $G$. The variety $\CP_K$ has a natural stratification into the projected Richardson varieties. The projected Richardson varieties in $\CP_K$ are the image of certain Richardson varieties in the full flag $\CB$ under the projection map $\pi: \CB \to \CP_K$. However, the combinatorial and geometric structures of the projected Richardson varieties in $\CP_K$ are more complicated than the Richardson varieties in $\CB$. 

Let $Q_K$ be the index set of the projected Richardson varieties in $\CP_K$ and $\preceq$ be the partial order on $Q_K$.  Williams  \cite{LW} showed that if $G$ is a connected reductive group, then the partial order set $Q_K$ has remarkable combinatorial properties: thinness, shellability, etc. Such combinatorial properties are used later by Galashin, Karp and Lam \cite{GKL} to prove the conjecture of Postnikov and Williams that the totally nonnegative part of $\CP_K$ is a regular CW complex.

\subsection{The Bruhat atlas of \cite{HKL}} One of the motivations in the unpublished work of Knutson, Lu and the second-named author \cite{HKL} is to use the Richardson varieties in the full flag variety $\tilde \CB$ of a ``much larger'' Kac-Moody group $\tilde G$ as the model for the decomposition of $\CP_K$ into projected Richardson varieties, and many other stratified spaces arising in Lie theory. Consequently, many remarkable combinatorial properties and geometric properties on these stratified spaces may be deduced directly from those on the Bruhat order of the Weyl group $\tilde W$ of $\tilde G$ and the Richardson varieties of $\tilde \CB$. 

By definition, a {\it Bruhat atlas} for a stratified space $M=\sqcup_y \mathring{M}_y$ consists of a large Kac-Moody group $\tilde{G}$ and an open covering $M=\cup U$, such that 
\begin{itemize}
    \item for each $U$, there exists an isomorphism of stratified spaces from $U$ to a Schubert cell in the flag variety $\tilde{\CB}$ of $\tilde{G}$;
    \item for each $y$ and $U$, $U \cap \mathring{M}_y$ is mapped isomorphically to an open Richardson variety of $\tilde \CB$.
\end{itemize}

The group $\tilde{G}$ is called the {\it atlas group} for this Bruhat atlas. 

The first example of a Bruhat atlas was constructed by Snider \cite{Sn} for Grassmannian with the positroid stratification. A Bruhat atlas for the full flag variety of a connected reductive group $G$ was constructed by Knutson, Woo, and Yong in \cite{KWY}. A Bruhat atlas for any partial flag variety of a connected reductive group and a Bruhat atlas for the wonderful compactification of a semisimple adjoint group was then constructed in \cite{HKL}. A different ``atlas'' for the partial flag varieties of a connected reductive group was constructed recently by Galashin, Karp and Lam in \cite{GKL} and by Huang in \cite{Hu1}. A Bruhat atlas for the wonderful compactification of the symmetric space $PSO(2n)/SO(n-1)$ was recently given by Huang in \cite{Hu2}.

\subsection{The Birkhoff-Bruhat atlas} In the works discussed above, the group $G$ involved is a connected reductive group, i.e. a Kac-Moody group of finite type. 

A Bruhat atlas for $\CP_K$ does not exists when $G$ is of infinite type, due to the following reasons. First, the partial flag variety $\CP_K$, in general, are infinite-dimensional. Thus one can not use Schubert cells (which is finite-dimensional) as an atlas for $\CP_K$. This gives a geometric obstruction. There is also a combinatorial obstruction arising from comparison of the partial orders. Note that the partial order $\preceq$ on $Q_K$ involved both the Bruhat order and the opposite Bruhat order in the Weyl group $W$ of $G$. By the definition of Bruhat atlas, one needs to embed $Q_K$ into the Weyl group $\tW$ of an atlas group. It is only possible if the Weyl group $W$ has the longest element, which interchanges the Bruhat order and the opposite Bruhat order on $W$. 

The main purpose of this paper is to introduce a suitable ``atlas model'' for the partial flag varieties of any Kac-Moody group. We use the open $J$-Schubert cells   in the full flag variety $\tilde \CB$ of an atlas group $\tilde G$ instead of the Schubert cells in $\tilde \CB$ as in the original definition of Bruhat atlas.

The decomposition of $\tilde \CB$ into the $J$-Schubert cell was introduced by Billig and Dyer in \cite{BD}, which simultaneously generalizes both the Bruhat decomposition of $\CB$ into the Schubert cells and the Birkhoff decomposition of $\CB$ into the opposite Schubert cells. A $J$-Schubert cell, in general, is neither finite dimensional nor finite codimensional. A Birkhoff-Bruhat atlas for $\CP_K$ consists of an atlas group $\tilde{G}$ and an open covering $\CP_K=\cup U$ such that 
\begin{itemize}
    \item  for each $U$, an embedding of $U$ into the full flag variety $\tilde \CB$ of $\tilde{G}$;
    \item the intersection of $U$ with any projected Richardson variety is mapped isomorphically to a $J$-Richardson variety in $\tilde \CB$. 
\end{itemize}

The main result of this paper is 

\begin{theorem}
    Any partial flag variety $\CP_K$ of a Kac-Moody group admits a Birkhoff-Bruhat atlas. 
\end{theorem}

It is also worth mentioning that even in the finite type case, the atlas groups from the Birkhoff-Bruhat atlas we constructed and those from the Bruhat atlas in \cite{HKL} are different. Thus our construction provides a new ``atlas model'' for the partial flag varieties for connected reductive groups. 

We also construct a combinatorial ``atlas model'' for the poset $Q_K$. This ``atlas model'' identifies the poset $(Q_K, \preceq)$ with a convex subset of the Weyl group $\tilde{W}$ of the atlas group $\tilde{G}$ with respect to a twisted Bruhat order $\leI$. This combinatorial ``atlas model'' is valid for $Q_K$ from an arbitrary Coxeter group. 

\begin{theorem}
    The partial order $\preceq$ on $Q_K$ is thin and EL-shellable. 
\end{theorem}

We refer to section \ref{sec:EL} for the definition of thinness and EL-Shellability. This result generalizes the previous work of Williams \cite{LW}. 

Galashin, Karp and Lam \cite[Conjecture~10.2]{GKL} conjectured  that the totally nonnegative part of $\CP_K$ is a regular CW complex for a Kac-Moody group $G$, and $Q_K$ is its face poset. Theorem~{1.2} confirms the combinatorial aspect of their conjecture.

\subsection{Organization} This paper is organized as follows. We recall preliminaries of Kac-Moody groups and J-Schubert cells in Section~\ref{sec:pre}. We then define a Birkhoff-Bruhat atlas and construct such an atlas for the partial flag variety $\CP_{K}$ of arbitrary type in Section~\ref{sec:BB}. We discuss some combinatorial consequences in Section~\ref{sec:EL}. We construct another Birkhoff-Bruhat atlas for the partial flag variety $\CP_K$ when is $K$ is of finite type in Section~\ref{sec:2nd-atlas}.  We discuss examples in Section~\ref{sec:example}.

\vspace{.2cm}
\noindent {\bf Acknowledgement: } We thank Thomas Lam and Lauren Williams for helpful comments.

HB is supported by a NUS startup grant. XH is partially supported by a start-up grant and by funds connected with
Choh-Ming Chair at CUHK, and by Hong Kong RGC grant 14300220.


\section{Preliminary}\label{sec:pre}
\subsection{Minimal Kac-Moody groups}
 
Let $I$ be a finite set and $A=(a_{ij})_{i, j \in I}$ be a symmetrizable generalized Cartan matrix in the sense of \cite[\S 1.1]{Kac}. A {\it Kac-Moody root datum} associated to $A$ is a quintuple $$\CD=(I, A, X, Y, (\a_i)_{i \in I}, (\a^\vee_i)_{i \in I}),$$ where $X$ is a free $\BZ$-module of finite rank with $\BZ$-dual $Y$, and the elements $\a_i$ of $X$ and $\a^\vee_i$ of $Y$ such that $\<\a^\vee_j, \a_i\>=a_{ij}$ for $i, j \in I$. We denote by $\omega_i \in X$  the element that $\<\a^\vee_j, \omega_i\>=\delta_{ij}$. We shall assume the root datum $\mathcal{D}$ is simply connected.

We have natural actions of $W$ on both $X$ and $Y$. Let 
$$
\Delta^{re} = \{w ( \pm \alpha_i) \in X \mid i \in I, w \in W\} \subset X
$$
 be the set of real roots. Then $\D^{re}=\D^{re}_+ \sqcup \D^{re}_-$ is the union of positive real roots and negative real roots. 

Let $\kk$ be an algebraically closed field. The {\it minimal Kac-Moody group} $G$ associated to the Kac-Moody root datum $\CD$ is the group generated by the torus $T=Y \otimes_\BZ \kk^\times$ and the root subgroup $U_\a \cong \kk$ for each real root $\a$, subject to the Tits relations \cite{Ti87}. Let $U^+  \subset G $ (resp. $U^-  \subset G$) be the subgroup generated by $U_\a$ for $\a \in \D^{re}_+$ (resp. $\a \in \D^{re}_-$). Let $B^{\pm}  \subset G $ be the Borel subgroup generated by $T$ and $U^{\pm} $. We fix a pinning of $G$ consisting of $(T, B^+, B^-, x_i, y_i; i \in I)$ with one parameter subgroups $x_i: \kk \rightarrow U_{\alpha_i}$ and $y_i: \kk \rightarrow U_{-\alpha_i}$ analogous to \cite{Lu94}. We have an anti-involution $\Psi$ of $G$ analogous to \cite[\S1.2]{Lu94} such that $\Psi(x_i(a)) = y_i(a)$, $\Psi(y_i(a)) = x_i(a)$ and $\Psi(t) =t$ for $a \in \kk$, $t \in T$.

Let $J \subset I$ (not necessarily of finite type). We denote by $P^+_J$ the subgroup of $G$ generated by $B^+$ and $U_{-{\a_j}}$ for $j \in J$. Let $W_J$ be the subgroup of $W$ generated by $\{s_j\}_{j \in J}$. Let $W^J$ be the set of minimal-length coset representatives of $W/W_J$ and ${}^J W$ be the set of minimal-length coset representatives of $W_J \backslash W$. For $i \in I$, we define 
\[
\dot{s}_i = x_{i}(-1) y_i (1) x_{i}(-1) \in G.
\]
For any $w \in W$ with reduced expression $w = s_{i_1} \cdots s_{i_n}$, we define 
\[
\dot{w} = \dot{s}_{i_1} \cdots \dot{s}_{i_n} \in G.
\]
It is known that $\dot{w}$ is well-defined and independent of the reduced expression. 

Let $L_J$ be the subgroup of $P^+_J$ generated by $T$, $U_{\pm{\a_j}}$ for $j \in J$. We denote by $\D^{re}_J =  \{w ( \pm \alpha_j) \in X \vert j \in J, w \in W_J\} \subset \Delta^{re}$ the set of real roots of $L_J$. We write $\D^{re}_{J,\pm} = \D^{re}_J \cap \Delta^{re}_{\pm}$.
We denote by $U_{P^+_J}$ the unipotent radical of  ${P^+_J}$, which is generated by $U_{{\a}}$ for $\a \in \D^{re}_+ - \D^{re}_J$. We have following Levi decomposition of $P^+_J$ \cite[Theorem~B.39]{Mar}
\begin{equation}\label{eq:levi}
  P^+_J = L_J \ltimes U_{P^+_J}.
\end{equation}

We similarly define the subgroup $P^-_J$ of $G^{\min}$ as the subgroup generated by $B^-$ and $U_{ {\a_j}} $ for $j \in J$, with the Levi decomposition 
\begin{equation}\label{eq:levi-}
	P^-_J = L_J \ltimes U_{P^-_J}.
\end{equation}



\subsection{The full flag variety}

In this subsection, we recall several results on the Kac-Moody flag varieties.  

We denote by $\CB$ the (thin) full flag variety \cite{Kum}, equipped with the ind-variety structure.  Let $v, w \in W$.  Define, respectively, the Schubert cell, the opposite Schubert cell and the open Richardson variety by 
\[
\mathring{X}^w =B^+ \dot{w} B^+/B^+, \quad \mathring{X}_v=B^- \dot{v} B^+/B^+, \quad \mathring{R}_{v, w}=\mathring{X}^w \cap \mathring{X}_v.
\]

We have the Bruhat decomposition $\CB=\sqcup_{w \in W} \mathring{X}^w$ and the Birkhoff decomposition $\CB=\sqcup_{v \in W} \mathring{X}_v$. It is known that $\mathring{R}_{v, w} \neq \emptyset$ if and only if $v \le w$. In this case, $\mathring{R}_{v, w}$ is irreducible of dimension $\ell(w)-\ell(v)$. We also have the decomposition $$\CB=\sqcup_{v \le w} \mathring{R}_{v, w}.$$

Let $X^w, X_v, R_{v, w}$ be the (Zariski) closure of $\mathring{X}^w, \mathring{X}_v, \mathring{R}_{v, w}$ respectively. By \cite[Proposition~7.1.15\&7.1.21]{Kum}, 
$$
X^w=\bigsqcup_{w' \le w} \mathring{X}^{w'}, \qquad X_v=\bigsqcup_{v' \ge v} \mathring{X}_{v'}.
$$

As the Schubert varieties and opposite Schubert varieties intersect transversally, we also have (see e.g. \cite{KLS}) 
$$
R_{v, w}=\bigsqcup_{v \le v' \le w' \le w} \mathring{R}_{v', w'}.
$$



\subsection{The $J$-Schubert cells and $J$-Richardson varieties}\label{sec:J-Schubert}
Let $J \subset I$. Following \cite[Closure patterns]{BD}, we define the partial order $\leJ$ on $W$ as follows. The partial order $\leJ$ on $W$ is generated by the relations $s_\b w \, {}^J\!\!<w$ for $w \in W$ and $\b \in \Psi_J$ with $w \i(\b) \in \D^{re, -}$. Define a (non-standard) length function ${}^J\ell$ on $W$ by 
$$
{}^J \ell(w)=\ell(w)-2 \sharp(\D^{re, +}_J \cap w \i (\D^{re, -})).
$$

Note that any element in $W$ can be written in a unique way as $x y$ for $x \in W_J$ and $y \in {}^J W$. We have ${}^J \ell(x y)=\ell(y)-\ell(x)$, where $\ell(\cdot)$ denotes the usual length function. We define 
\[
\Psi^{+}_J=\D^{re, -}_J \sqcup (\D^{re, +}-\D^{re, +}_J), \qquad \Psi^{-}_J=\D^{re, +}_J \sqcup (\D^{re, -}-\D^{re, -}_J).
\]

\begin{remark}
The relation between our partial order $\leJ$ and the partial order $\le_{\Psi_J}$ used in \cite[Closure pattern]{BD} is the following 
\[
 v \leJ w \text { if and only if }  v^{-1} \le_{\Psi_J} w^{-1}. 
\]
Note that  $v \leJ w$ is not equivalent to $v^{-1} \leJ w^{-1}$ in general \cite[Page 18]{BD}.
\end{remark}
\begin{remark}
Let $W_J$ be a finite Weyl group. In this case, we denote by $w_J$ the longest element of $W_J$. Then for $w, w' \in W$, $w' \leJ w$ if and only if $w_J w' \le w_J w$. Moreover, we have ${}^J \ell(w)=\ell(w_J w)-\ell(w_J)$. 
\end{remark}


Let ${}^J B^+$ be the subgroup of $G$ generated by $T$ and $U_{\a}$ for $\a \in \Psi_J^+$. Then ${}^J B^+$ is the opposite Borel subgroup of the standard parabolic subgroup $P^+_J$. Similarly, let ${}^J B^-$ be the subgroup of $G$ generated by $T$ and $U_{\a}$ for $\a \in \Psi_J^-$. In the case where $J=\emptyset$, we have ${}^J B^+=B^+$ and ${}^J B^-=B^-$. In the case where $J=I$, we have ${}^J B^+=B^-$ and ${}^J B^-=B^+$. Let $B^{\pm}_J=L_J \cap B^{\pm}$ and $U^{\pm}_J=L_J \cap U^{\pm}$.

Thanks to the Levi decompositions \eqref{eq:levi}\&\eqref{eq:levi-}, we have 
\[
{}^J B^+ = B^-_J \ltimes U_{P^+_J}, \qquad {}^J B^- = B^+_J \ltimes U_{P^-_J}.
\]

Let $v, w \in W$. Define, respectively, the $J$-Schubert cell, the opposite $J$-Schubert cell and the open $J$-Richardson variety by 
$$
{}^J \mathring{X}^w ={}^J B^+  \dot{w} B^+/B^+, \quad {}^J \mathring{X}_v={}^J B^-  \dot{v} B^+/B^+,\quad {}^J \mathring{R}_{v, w}={}^J \mathring{X}^w \cap {}^J \mathring{X}_v.
$$

\begin{lem}\label{lem:JXw}
We have isomorphisms
\begin{gather*}  
(U_J^{-}  \cap \dot w U^- \dot w^{-1}) \times (U_{P^+_J}   \cap \dot w U^- \dot w^{-1} ) \rightarrow {}^J \mathring{X}^w, \quad
(x_1, x_2) \mapsto x_1 x_2  \dot{w} B^+/B^+; \\
(  U_J^{+}  \cap \dot v U^- \dot v^{-1}) \times (  U_{P^-_J}   \cap \dot v U^- \dot v^{-1}) \longrightarrow {}^J \mathring{X}_v, \quad
(x_1, x_2) \mapsto x_1 x_2  \dot{v} B^+/B^+.
\end{gather*}
\end{lem}
 
 \begin{proof}We prove the first statement here. The second one is entirely similar.  It follows from the Levi decompositions and \cite[Theorem~5.2.3]{Kum} that we have isomorphisms 
 \[
(U_J^{-}  \cap \dot w U^- \dot w^{-1}) \times  (U_J^{-} \cap \dot w U^+ \dot w^{-1}) \longrightarrow  U_J^{-},
 \]
 \[
(U_{P^+_J}   \cap \dot w U^- \dot w^{-1}) \times  (U_{P^+_J}   \cap \dot w U^+ \dot w^{-1}) \longrightarrow  U_{P^+_J}.
 \]
 
Therefore we have 
\[{}^J B^+  \dot{w} B^+/B^+ =  (U_J^{-}  \cap \dot w U^- w^{-1})  \cdot (U_{P^+_J}  \cap \dot w U^- w^{-1}) \cdot \dot{w} B^+/B^+. 
\]
 Now the lemma follows from the restriction of the isomorphism
 \[
    \dot w U^- \dot w^{-1} \longrightarrow \dot{w} U^- B^+ / B^+, \quad g \mapsto g \dot{w} B^+ /B^+. \qedhere
 \]
 \end{proof}
 
By \cite[Theorem 1]{BD}, we have 
\begin{equation}\label{eq:BD1}
\CB=\sqcup_{w \in W} {}^J\! \mathring{X}^w=\sqcup_{v \in W} {}^J\! \mathring{X}_v.
\end{equation}

Let ${}^J X^w$ and ${}^J X_v$ be the (Zariski) closure of ${}^J\! \mathring{X}^w$ and ${}^J\! \mathring{X}_v$ respectively. By \cite[Theorem 4]{BD}, we have 
\begin{equation}\label{eq:closure} {}^J X^w=\overline{{}^J \mathring{X}^w} = \bigsqcup_{w' \leJ w} \mathring{X}^{w'}, {}^J X_v=\overline{{}^J \mathring{X}_v} =\bigsqcup_{v \leJ v'} \mathring{X}_{v'}.
\end{equation}

\begin{proposition}\label{prop:intersection}
	Let $v, w \in W$. Then the following conditions are equivalent: 
	
	(1) ${}^J \mathring{X}^w \cap {}^J \mathring{X}_v \neq \emptyset$; 
	
	(2) ${}^J X^w \cap {}^J X_v \neq \emptyset$; 
	
	(3) $v \leJ w$. 
	
\end{proposition}

\begin{proof}
	It is obvious that $(1) \Rightarrow (2)$. 
	
	We show that $(2) \Rightarrow (3)$. If ${}^J X^w \cap {}^J X_v \neq \emptyset$, then there exists $z \in W$, such that ${}^J X^w \cap {}^J X_v \cap X^z \neq \emptyset$. Since $X^z$ is finite dimensional, ${}^J X^w \cap {}^J X_v \cap X^z$ is still finite dimensional. It is projective and stable under the left action of $T$. By \cite[Exercise 7.1.E.5]{Kum}, ${}^J X^w \cap {}^J X_v \cap X^z$ has a $T$-fixed point. Hence ${}^J X^w \cap {}^J X_v$ has a $T$-fixed point. 
	
	Note that the $T$-fixed points in $X$ are $\{\dot w B^+/B^+; w \in W\}$. Thus by \eqref{eq:BD1}, $\dot w B^+/B^+$ is the only $T$-fixed point of ${}^J \mathring{X}^w$ and of ${}^J \mathring{X}_w$. By \eqref{eq:closure}, the $T$-fixed points in ${}^J X^w$ are $\{\dot w' B^+/B^+; w' \leJ w\}$ and the $T$-fixed points in ${}^J X_v$ are $\{\dot v' B^+/B^+; v \leJ v'\}$. Since ${}^J X^w$ and ${}^J X_v$ have a common $T$-fixed point, we must have $v \leJ w$. 
	
	We then show that $(3) \Rightarrow (1)$. Let $U_v =  \dot v  U^- B^+/ B^+ \subset \CB$. Thanks to the Levi decomposition and \cite[Theorem~5.2.3]{Kum}, we have the isomorphisms
\[
	\begin{tikzcd}[row sep=tiny, column sep=small]
	(U_J^{-}U_{P^+_J}  \cap \dot v U^- \dot v^{-1}) \times (  U_J^{+}U_{P^-_J}  \cap \dot v U^- \dot v^{-1}) \ar[r, "\sim"]&\dot v U^- \dot v^{-1} \ar[r, "\sim"] &   U_v,\\
	(g_1, g_2) \ar[r,mapsto]& g_1 g_2 \ar[r,mapsto]&g_1 g_2 \dot{v} B^+/ B^+.
	\end{tikzcd}
\]

	By Lemma~\ref{lem:JXw}, we have the isomorphism 
	\[
	(U_J^{-}U_{P^+_J}  \cap \dot v U^- \dot v^{-1}) \times  {}^J \mathring{X}_v \xrightarrow{\sim} U_v.
	\]
	Since ${}^J \mathring{X}^w$ is $U_J^{-}U_{P^+_J}$-stable, we have, via restriction,
	\[
	(U_J^{-}U_{P^+_J}  \cap \dot v U^- \dot v^{-1}) \times  ({}^J \mathring{X}^w \cap {}^J \mathring{X}_v)  \xrightarrow{\sim} {}^J \mathring{X}^w \cap U_v.
	\]
	It remains to prove ${}^J \mathring{X}^w \cap U_v \neq \emptyset$ when $v \leJ w$. Since $U_v$ is open in $\CB$,  we have ${}^J \mathring{X}^w \cap U_v \neq \emptyset$ if and only if ${}^J {X}^w \cap U_v  \neq \emptyset$. Thanks to \eqref{eq:closure}, we have $\dot{v} B^+/ B^+ \in {}^J {X}^w$ if $v \leJ w$. We clearly have $\dot{v} B^+/ B^+ \in U_v$.  Therefore $\dot{v} B^+/ B^+ \in U_v \cap {}^J {X}^w$. The claim follows.
\end{proof}


\subsection{J-Schubert decompositions}In this subsection, we study a more refined version of Lemma~\ref{lem:JXw}.

For a group $K$ and subsets $K'$, $K_1$, $K_2$, $\dots$, $K_n$, we write 
\[
K' = K_1 \odot K_2 \odot \cdots \odot K_n
\]
 if any element $k' \in K'$ can be written uniquely as $k' = k_1 k_2 \cdots k_n$ with $k_i \in K_i$.
\begin{lem}\label{lem:+LL}
Let $v \in W_{J_1}$ and $w \in W_{J_2} \cap {}^{J_1} W$. We have 
\[
U^-_{J_1} U_{P^+_{J_1}}  \cap (\dot v \dot w) U^-    (\dot v \dot w)^{-1}  = (U^-_{J_1} \cap  \dot v U^-_{J_1}  \dot v^{-1}) \odot  \dot v(U^+_{P^+_{J_1}} \cap  \dot wU^-_{J_2} \dot w^{-1} ) \dot v^{-1}.
\]
\end{lem}

\begin{proof}Note that $\ell (vw) =\ell(v) + \ell(w)$. We recall the following decompositions from \cite[Theorem~5.2.3]{Kum}:
\begin{align}
   (\dot v \dot{w}) U^-   (\dot v \dot{w})^{-1} & = (U^- \cap  (\dot v \dot{w}) U^-    (\dot v \dot{w})^{-1}) \odot (U^+ \cap  (\dot v \dot{w}) U^-  (\dot v \dot{w})^{-1}) \tag{$\diamondsuit1$};\\
U^- \cap  \dot v U^-  \dot v^{-1}  &=  \dot v(U^- \cap  \dot wU^+ \dot w^{-1}) \dot v^{-1} \odot (U^- \cap  (\dot v \dot{w}) U^- (\dot v \dot{w}) ^{-1}); \tag{$\diamondsuit2$}\\
U^+ \cap  (\dot v \dot{w}) U^-  (\dot v \dot{w})^{-1}  &=  \dot v(U^+ \cap  \dot wU^- \dot w^{-1}) \dot v^{-1} \odot (U^+ \cap  \dot v U^-  \dot v^{-1}) \tag{$\diamondsuit3$}.
\end{align}

Thanks to the Levi decomposition 
$
P^+_{J_1} = L_{J_1}  \ltimes U_{P^+_{J_1}}
$,
the decompositions are compatible with the restriction from $U^\pm$ to $U^\pm_{J_1}$ as well as from $U^+$ to $U_{P^+_{J_1}}$.

It follows from $(\diamondsuit1)$ that 
\[
U^-_{J_1} U_{P^+_{J_1}}  \cap  (\dot v \dot{w}) U^-  (\dot v \dot{w})^{-1} = (U^-_{J_1}  \cap  (\dot v \dot{w}) U^-  (\dot v \dot{w})^{-1}) \odot (U_{P^+_{J_1}} \cap  (\dot v \dot{w}) U^-  (\dot v \dot{w})^{-1}).
\]

Since $w \in {}^{J_1} W$, we have $U^-_{J_1}  \cap  \dot wU^+ \dot w^{-1} = \{e\}$. 
Therefore it follows from $(\diamondsuit2)$ that
\[
U^-_{J_1}  \cap  \dot v U_{J_1}^-  \dot v^{-1} = U^-_{J_1}  \cap  \dot v U^-  \dot v^{-1}= U^-_{J_1}  \cap  (\dot v \dot{w}) U^-  (\dot v \dot{w})^{-1}.
\]

Finally, since $v \in W_{J_1}$, we have $U_{P^+_{J_1}} \cap  \dot v U^-  \dot v^{-1} = \{ e \}$. Since $   w \in W_{J_2}$, it follows from $(\diamondsuit3)$ that 
\[
U_{P^+_{J_1}}  \cap (\dot v \dot{w}) U^-  (\dot v \dot{w})^{-1} =  \dot v(U_{P^+_{J_1}}  \cap  \dot wU^- \dot w^{-1}) \dot v^{-1} =  \dot v(U_{P^+_{J_1}}  \cap  \dot wU_{J_2}^- \dot w^{-1}) \dot v^{-1}
\]
The lemma follows. 
\end{proof}

The following result follows easily from Lemma~\ref{lem:JXw} and Lemma~\ref{lem:+LL}.

\begin{cor}\label{cor:JXw}
Let $J_1, J_2 \subset I$. Let $v \in W_{J_1}$ and $w \in W_{J_2} \cap {}^{J_1} W$. We have an isomorphism 
\begin{align*}
 (U^-_{J_1} \cap  \dot v U^-_{J_1}  \dot v^{-1}) \times   \dot v(U^+_{P^+_{J_1}} \cap  \dot wU^-_{J_2} \dot w^{-1} ) \dot v^{-1} &\longrightarrow {}^{J_1}\! \mathring{X}^{vw}, \quad (g_1, g_2)& \mapsto g_1g_2  (\dot v \dot{w})B^+.
\end{align*}
\end{cor}

\begin{lem}\label{lem:J1B}
Let $v \in W_{J_1}$ and $w \in W_{J_2} \cap {}^{J_1} W$. We have 
\begin{align*}
U^+_{J_1} U_{P^-_{J_1}} \cap  (\dot v \dot w) U^-  (\dot v \dot w)^{-1} 
= &(U^+_{J_1} \cap  \dot v U^-_{J_1}  \dot v^{-1}) \odot   \dot v(U^-_{{J_2}} \cap U_{P^-_{J_1}}  \cap  \dot w U_{J_2}^-  \dot w^{-1}) \dot v^{-1} \\
 &  \odot  \dot v(U_{P^-_{J_2}} \cap U_{P^-_{J_1}}  \cap  \dot w U^-  \dot w^{-1})  \dot v^{-1}.
\end{align*}
\end{lem}
\begin{proof}
Note that $\ell (vw) =\ell(v) + \ell(w)$. We recall again the following decompositions from \cite[Theorem~5.2.3]{Kum}:
\begin{align}
   (\dot v \dot{w}) U^-  (\dot v \dot{w})^{-1} & =  (U^+ \cap  (\dot v \dot{w}) U^-  (\dot v \dot{w})^{-1})  \odot (U^- \cap  (\dot v \dot{w}) U^-  (\dot v \dot{w})^{-1}) \tag{$\diamondsuit1$};\\
U^+ \cap  (\dot v \dot{w}) U^-  (\dot v \dot{w}) ^{-1}  &=  \dot v(U^+ \cap  \dot wU^- \dot w^{-1}) \dot v^{-1} \odot (U^+ \cap  \dot v U^-  \dot v^{-1}) \tag{$\diamondsuit3$}.
\end{align}
 Thanks to the Levi decomposition 
$
P^+_{J_1} = L_{J_1}  \ltimes U_{P^+_{J_1}}
$,
the decompositions above are compatible with the restriction from $U^\pm$ to $U^\pm_{J_1}$ as well as from $U^+$ to $U_{P^+_{J_1}}$.

It follows from $(\diamondsuit1)$ that 
\[
U^+_{J_1} U_{P^-_{J_1}}  \cap  (\dot v \dot{w})  U^-  (\dot v \dot{w})^{-1} = (U^+_{J_1}  \cap  (\dot v \dot{w}) U^-  (\dot v \dot{w})^{-1}) \odot (U_{P^-_{J_1}} \cap  (\dot v \dot{w}) U^-  (\dot v \dot{w})^{-1}).
\]

Since $w \in  W_{J_2} \cap {}^{J_1} W$ and $v \in W_{J_1}$, it follows from $(\diamondsuit3)$ that
\[
U^+_{J_1}  \cap  (\dot v \dot{w}) U^-  (\dot v \dot{w})^{-1} = U^+_{J_1}  \cap  \dot v U^-  \dot v^{-1} = U^+_{J_1}  \cap  \dot v U_{J_1}^-  \dot v^{-1}.
\]
Since $v \in W_{J_1}$, we have 
\[
U_{P^-_{J_1}} \cap  (\dot v \dot{w}) U^-  (\dot v \dot{w})^{-1} = \dot  v(U_{P^-_{J_1}} \cap  \dot w U^-  \dot w^{-1}) \dot v^{-1}.
\]

Thanks to the Levi decomposition of $P^-_{J_2}$, we further have 
\begin{align*}
& U_{P^-_{J_1}} \cap  (\dot v \dot{w}) U^-  (\dot v \dot{w})^{-1} \\
= & \,  \dot v(U_{P^-_{J_1}} \cap  \dot w U^-  \dot w^{-1}) \dot v^{-1} \\
= &\,  \dot v(U^-_{{J_2}} \cap U_{P^-_{J_1}}  \cap  \dot w U^-  \dot w^{-1}) \dot v^{-1} \odot  \dot v(U_{P^-_{J_2}} \cap U_{P^-_{J_1}}  \cap   \dot{w} U^-   \dot{w}^{-1})  \dot v^{-1} \\
= &\,  \dot v(U^-_{{J_2}} \cap U_{P^-_{J_1}}  \cap  \dot w U_{J_2}^-  \dot w^{-1}) \dot v^{-1} \odot  \dot v(U_{P^-_{J_2}} \cap U_{P^-_{J_1}}  \cap   \dot{w} U^-   \dot{w}^{-1})  \dot v^{-1}.
\end{align*}
The lemma is proved.
\end{proof}

\subsection{Product of Parabolic subgroups}
Let $J_1, J_2 \subset I$. We study the decomposition of 
$
 P^+_{J_1} P^+_{J_2} /B^+ 
$
with respect to the $J_1$-Schubert cells and the opposite $J_1$-Schubert cells.  

We first consider the decomposition into the $J_1$-Schubert cells. 

\begin{proposition}\label{prop:LL1}
Let $J_1, J_2 \subset I$. Then we have 
$$
P^+_{J_1} P^+_{J_2} = L_{J_1} L_{J_2} B^+=\bigsqcup_{\tw \in W_{J_1} W_{J_2}}\, {}^{J_1} \!B^+ \dot \tw B^+.
$$
\end{proposition}

\begin{proof}
It follows from \eqref{eq:BD1} that $\bigcup_{\tw \in W_{J_1} W_{J_2}} \,{}^{J_1} \!B^+ \dot \tw B^+$ is a   disjoint union. We have 
\begin{align*} 
L_{J_2} B^+ 
&=\bigsqcup_{w \in W_{J_2} \cap {}^{J_1} W} (L_{J_2} \cap P_{J_1}) \dot w B^+\\
 &=\bigsqcup_{w \in W_{J_2} \cap {}^{J_1} W} (L_{J_2} \cap L_{J_1})  (L_{J_2} \cap U_{P^+_{J_1}}) \dot w B^+.
\end{align*} 
Thus
\begin{align*} 
L_{J_1} L_{J_2} B^+ &=\bigcup_{w \in W_{J_2} \cap {}^{J_1} W} L_{J_1} (L_{J_2} \cap U_{P^+_{J_1}}) \dot w B^+ \\ 
&=\bigcup_{v \in W_{J_1}, w \in W_{J_2} \cap {}^{J_1} W} (L_{J_1} \cap U^-) \dot v (L_{J_1} \cap U^+) (L_{J_2} \cap U_{P^+_{J_1}}) \dot w B^+ \\ 
& \subset \bigcup_{v \in W_{J_1}, w \in W_{J_2} \cap {}^{J_1} W} (L_{J_1} \cap U^-) \dot v  U_{P^+_{J_1}}  (L_{J_1} \cap U^+) \dot w B^+ \\ 
&=\bigcup_{v \in W_{J_1}, w \in W_{J_2} \cap {}^{J_1} W} (L_{J_1} \cap U^-) U_{P^+_{J_1}} \dot v  \dot w \bigl(\dot w \i (L_{J_1} \cap U^+) \dot w) B^+ \\ 
&=\bigcup_{v \in W_{J_1}, w \in W_{J_2} \cap {}^{J_1} W} {}^{J_1} B^+ \dot v \dot w B^+.
\end{align*} 

Now we proceed with the reverse inclusion. Note that any element in $W_{J_1} W_{J_2}$ can be written in a unique way as $vw$ for some $v \in W_{J_1}$ and $w \in W_{J_2} \cap {}^{J_1} W$. By definition  and the Levi decomposition \eqref{eq:levi}, ${}^{J_1} B^+=(L_{J_1} \cap U^-) U_{P^+_{J_1}} T$. Since $v \in W_{J_1}$, the conjugation action of $\dot v$ stabilizes $U_{P^+_{J_1}} T$. Thus 
$$
{}^{J_1} B^+ \dot v \dot w B^+=(L_{J_1} \cap U^-) \dot v (U_{P^+_{J_1}} \dot w B^+)=(L_{J_1} \cap U^-) \dot v(U_{P^+_{J_1}} \cap \dot w B^- \dot w \i) \dot w B^+.
$$
Since $w \in W_{J_2}$, we have $U_{P^+_{J_1}} \cap \dot w B^- \dot w \i \subset U_{P^+_{J_1}} \cap L_{J_2}$. Thus 
$$
{}^{J_1} B^+ \dot v \dot w B^+ \subset (L_{J_1} \cap U^-) \dot v (L_{J_2} \cap U^+) \dot w B^+ \subset L_{J_1} L_{J_2} B^+.
$$
The statement is proved. 
\end{proof}

We then consider the decomposition into opposite $J_1$-Schubert cells. 
\begin{prop}\label{prop:LL}
We have 
$$
P^+_{J_1} P^+_{J_2} =\bigsqcup_{\tw \in W_{J_1} W_{J_2}}\, {}^{J_1} \!B^- \dot \tw B^+ \cap P^+_{J_1} P^+_{J_2}.
$$
\end{prop}

\begin{proof}

Note that $P^-_{J_1} \cap L_{J_2}$ is an opposite parabolic subgroup of $L_{J_2}$. By the Birkhoff decomposition of $L_{J_2}$ and the Levi decomposition of $P^+_{J_2}$, we have 
\[
P^+_{J_2}  = \bigsqcup_{w \in W_{J_2} \cap {}^{J_1}W} (P^{-}_{J_1} \cap L_{J_2})   \dot w B^+=\bigsqcup_{w \in W_{J_2} \cap {}^{J_1}W} (L_{J_1} \cap L_{J_2}) (U^{-}_{J_2} \cap U_{P^{-}_{J_1}})   \dot w B^+.
\]

Hence 
\begin{align*}
P^+_{J_1} P^+_{J_2} &
= \bigcup_{w \in W_{J_2} \cap {}^{J_1}W}  L_{J_1} (U^{-}_{J_2} \cap U_{P^{-}_{J_1}})   \dot w B^+\\
& =  \bigcup_{v \in W_{J_1}, w \in W_{J_2} \cap {}^{J_1}W}  U^+_{J_1}  \dot v U^+_{J_1}  (U^{-}_{J_2} \cap U_{P^{-}_{J_1}})   \dot w B^+ \\
& = \bigcup_{v \in W_{J_1}, w \in W_{J_2} \cap {}^{J_1}W}  (U^+_{J_1} \cap  \dot v U^-_{J_1}  \dot v^{-1} )  \dot v U^+_{J_1}  (U^{-}_{J_2} \cap U_{P^{-}_{J_1}})   \dot w B^+. \\ 
\end{align*}

We show that 

(a) $U^+_{J_1}  (U^{-}_{J_2} \cap U_{P^{-}_{J_1}})  \dot wB^+   =  (U^{-}_{J_2} \cap U_{P^{-}_{J_1}})  \dot wB^+$.

We have the decomposition $U^+_{J_1} = (U^+_{J_1}\cap U^{+}_{J_2})(U^+_{J_1}\cap U_{P^+_{J_2}})$.
Since $(U^{-}_{J_2} \cap U_{P^{-}_{J_1}})  \dot w \subset L_{J_2}$, we have $
 U_{P^{+}_{J_2}} (U^{-}_{J_2} \cap U_{P^{-}_{J_1}})   \dot w =  (U^{-}_{J_2} \cap U_{P^{-}_{J_1}})   \dot w  U_{P^{+}_{J_2}}$ and
\[
 (U^+_{J_1}\cap U^{+}_{J_2})  (U^+_{J_1}\cap U_{P^{+}_{J_2}}) (U^{-}_{J_2} \cap U_{P^{-}_{J_1}})   \dot wB^+  
= (U^+_{J_1}\cap U^{+}_{J_2})    (U^{-}_{J_2} \cap U_{P^{-}_{J_1}})  \dot w B^+.
\]
Thanks to the Levi decomposition of $L_{J_2} \cap P^-_{J_1}$, we further have 
\[
 (U^+_{J_1}\cap U^{+}_{J_2})    (U^{-}_{J_2} \cap U_{P^{-}_{J_1}})=  (U^{-}_{J_2} \cap U_{P^{-}_{J_1}})(U^+_{J_1}\cap U^{+}_{J_2}).
\] 
Since $w \in {}^{J_1}W$, we have $(U^+_{J_1}\cap U^{+}_{J_2})      \dot w B^+ =  \dot wB^+$. 

Thus (a) is proved.

Now we have 
\begin{equation}\label{eq:PP}
P^+_{J_1} P^+_{J_2}    = \bigcup_{v \in W_{J_1}, w \in W_{J_2} \cap {}^{J_1}W}  (U^+_{J_1} \cap  \dot v U^-_{J_1} v^{-1} )  \dot v   (U^{-}_{J_2} \cap U_{P^{-}_{J_1}})   \dot w B^+.
\end{equation}

Since $  \dot v   (U^{-}_{J_2} \cap U_{P^{-}_{J_1}})  \subset U_{P^-_{J_1}}  \dot v$, we have 
\[
(U^+_{J_1} \cap  \dot v U^-_{J_1}  \dot v^{-1} )  \dot v   (U^{-}_{J_2} \cap U_{P^{-}_{J_1}})   \dot w B^+ \subset \, {}^{J_1} \!B^-  \dot v   \dot w B^+.\]
Thanks to \eqref{eq:BD1}, we have 
\[
P^+_{J_1} P^+_{J_2} \subset \bigsqcup_{v \in W_{J_1}, w \in W_{J_2} \cap {}^{J_1}W}  \, {}^{J_1} \!B^-  \dot v   \dot w B^+.
\]

Then we have that $(U^+_{J_1} \cap  \dot v U^-_{J_1}  \dot v^{-1} ) \dot  v   (U^{-}_{J_2} \cap U_{P^{-}_{J_1}})   \dot w B^+=\, {}^{J_1} \!B^-  \dot v   \dot w B^+ \cap P^+_{J_1} P^+_{J_2}$
and 
\[
P^+_{J_1} P^+_{J_2}=\bigsqcup_{v \in W_{J_1}, w \in W_{J_2} \cap {}^{J_1}W}  \, ({}^{J_1} \!B^-  \dot v   \dot w B^+ \cap P^+_{J_1} P^+_{J_2}).
\]
Since $W_{J_1} W_{J_2}=W_{J_1} (W_{J_2} \cap {}^{J_1}W)$, we have $P^+_{J_1} P^+_{J_2} =\bigsqcup_{\tw \in W_{J_1} W_{J_2}}\, {}^{J_1} \!B^- \dot \tw B^+ \cap P^+_{J_1} P^+_{J_2}$ and the proposition is proved. 
\end{proof}

Combining Proposition~\ref{prop:LL} and Lemma~\ref{lem:J1B}, we have the following proposition.
\begin{prop}\label{prop:-LL}
Let $J_1$, $J_2 \subset I$ and $v \in W_{J_1}$, $w \in W_{J_2} \cap {}^{J_1}W$. We have an isomorphism 
\[
(U^+_{J_1} \cap  \dot v U^-_{J_1}  \dot v^{-1} )  \times  \big(  \dot v(U^{-}_{J_2} \cap U_{P^{-}_{J_1}} \cap  \dot w U^{-}_{J_2}  \dot w^{-1})   \dot v^{-1} \big) \rightarrow \, {}^{J_1} \mathring{X}_{vw} \cap P^+_{J_1} P^+_{J_2} / B^+. 
\]
\end{prop}

\begin{proof}
By Proposition~\ref{prop:LL}, we have
\begin{align*}
 \, {}^{J_1} \mathring{X}_{vw}\cap P^+_{J_1} P^+_{J_2} / B^+ &=  (U^+_{J_1} \cap  \dot v U^-_{J_1}  \dot v^{-1} ) \dot  v   (U^{-}_{J_2} \cap U_{P^{-}_{J_1}})   \dot w B^+ \\
 &= (U^+_{J_1} \cap  \dot v U^-_{J_1}  \dot v^{-1} ) \dot  v  (U^{-}_{J_2} \cap U_{P^{-}_{J_1}} \cap  \dot w U^{-}_{J_2}  \dot w^{-1})   \dot w B^+.
\end{align*}

By Lemma~\ref{lem:JXw}, we have the isomorphism
\[
\big(  U_{J_1}^{+}   \cap (\dot v \dot w) U^- (\dot v \dot w)^{-1}\big) \times \big(  U_{P^-_{J_1}}   \cap (\dot v \dot w) U^- (\dot v \dot w)^{-1}\big) \xrightarrow{\sim} {}^{J_1} \mathring{X}_{vw}.
\]
Thanks to Lemma~\ref{lem:J1B}, the lemma follows from the restriction of the above isomorphism to
\[
(U^+_{J_1} \cap  \dot v U^-_{J_1}  \dot v^{-1} )  \times  \big(  \dot v(U^{-}_{J_2} \cap U_{P^{-}_{J_1}} \cap  \dot w U^{-}_{J_2}  \dot w^{-1})   \dot v^{-1} \big) \xrightarrow{\sim} \, {}^{J_1} \mathring{X}_{vw} \cap P^+_{J_1} P^+_{J_2} / B^+. \qedhere
\]
\end{proof}


\section{A Birkhoff-Bruhat atlas}\label{sec:BB}

\subsection{Definitions} Let $M$ be an ind-variety over $\kk$. A {\it stratification} on $M$ is a family of locally closed, finite dimensional subvarieties $\{\mathring{M}^y\}_{y \in \CY}$ indexed by a   poset $\CY$ such that 

\begin{itemize}
	\item $M=\sqcup_{y \in \CY} \mathring{M}_y$; 
	
	\item For any $y \in \CY$, the Zariski closure $M^y$ of $\mathring{M}_y$ equals $\sqcup_{y' \le y} \mathring{M}_{y'}$. 
\end{itemize}

Assume furthermore that the minimal strata in the stratification $\CY_{\min}$ of $M$ are points. A Birkhoff-Bruhat atlas on $(M, \CY)$ is the following data: 

\begin{enumerate}
	\item an open covering for $M$ consisting of open sets $U_f$ around the minimal strata $f \in \CY$;
	
	\item a (minimal) Kac-Moody group $\tilde G $ and a subset $J$ of the set of simple roots of $\tilde G$; 
	
	\item for any minimal stratum $f \in\CY$, an embedding $c_f$ from $U_f$ into the flag variety $\tilde  \CB$ of $\tilde  G $ such that $c_f(U_f \cap \mathring{M}_y)$ is an open $J$-Richardson variety of $\tilde  \CB$ for any $y \in \tilde  \CY$. 
	 
\end{enumerate}



\subsection{Partial flag varieties} \label{sec:QK}

Let $K \subset I$ and $\CP_K=G/P^+_K$ be the partial flag variety. Then we have the decomposition $$
\CP_K=\sqcup_{w \in W^K} B^+   \dot w P^+_K/P^+_K=\sqcup_{v \in W^K} B^-    \dot v P^+_K/P^+_K.
$$

Let $Q_K=\{(v, w) \in W \times W^K \mid v \le w\}$. Define the partial order $\preceq$ on $Q_K$ as follows: 
\[
 (v', w') \preceq (v, w)  \text{ if there exists } u \in W_K \text{ such that } v \le v' u \le w' u \le w. 
\]

For any $(v, w) \in Q_K$, set 
$$\mathring{\Pi}_{v, w}=\pi_K(\mathring{R}_{v, w}) \text{ and } \Pi_{v, w}=\pi_K(R_{v, w}),$$ where $\pi_K: \CB \to \CP_K$ is the projection map. Then $\Pi_{v, w}$ is the (Zariski) closure of $\mathring{\Pi}_{v, w}$ in $\CP_K$. We call $\mathring{\Pi}_{v, w}$ an {\it open projected Richardson variety} and $\Pi_{v, w}$ a {\it closed projected Richardson variety}.

By \cite[Proposition 3.6]{KLS}, we have 
\begin{equation}\label{eq:Richardson}
\CP_K=\sqcup_{(v, w) \in Q_K} \mathring{\Pi}_{v, w} \text{ and } \Pi_{v, w}=\sqcup_{(v', w') \in Q_K; (v', w') \preceq (v, w)} \mathring{\Pi}_{v', w'}.
\end{equation}

	\vspace{.2cm}
	{\bf Let $K\subset I$. The goal of the rest of this section is devoted to construct an Birkhoff-Bruhat atlas for the stratified space $M=\CP_K$ with the stratification $\{\mathring{\Pi}_{v, w}\}_{(v, w) \in Q_K}$ considered in \S\ref{sec:QK}}.
	\vspace{.2cm}

\subsection{The Kac-Moody group $\tilde{G}$}\label{sec:tildeG}
	We construct the set of simple roots and the associated generalized Cartan matrix of the Kac-Moody $\tilde{G}$ from the original Kac-Moody group ${G}$. We list some examples of such construction in \S\ref{sec:example}.

The set $\tilde{I}$ of simple roots is the union of two copies of $I$, glued along $K$. More precisely, let $I^\flat=\{i^\flat \mid i \in I\}$ and $I^\sharp=\{i^\sharp \mid i \in I\}$ be the two copies of $I$. Then $\tilde I=I^\flat \cup I^\sharp$ with $I^\flat \cap I^\sharp=\{k^\flat=k^\sharp \mid k \in K\}$. For any $i \in I$, we set $(i^\flat)^\natural=(i^\sharp)^\natural=i$. The generalized Cartan matrix $\tilde A=(\tilde a_{\tilde i, \tilde i'})_{\tilde i, \tilde i' \in \tilde I}$ is defined as follows: 
\begin{itemize}
	\item for $\tilde i, \tilde i' \in I^\flat$, $\tilde a_{\tilde i, \tilde i'}=a_{\tilde i^\natural, (\tilde i')^\natural}$; 
	
	\item for $\tilde i, \tilde i' \in I^\sharp$, $\tilde a_{\tilde i, \tilde i'}=a_{\tilde i^\natural, (\tilde i')^\natural}$; 
	
	\item for $\tilde i \in \tilde I-I^\flat$ and $\tilde i' \in \tilde I-I^\sharp$, $\tilde a_{\tilde i, \tilde i'}=\tilde a_{\tilde i', \tilde i}=0$. 
\end{itemize}

Since $I^\sharp \cap I^\flat=\{k^b=k^\sharp \mid k\in K\}$, we have $\tilde a_{k^\flat, (k')^\flat}=a_{k, k'}=\tilde a_{k^\sharp, (k')^\sharp}$ and the generalized Cartan matrix $\tilde A$ is well-defined. The generalized Cartan matrix $\tilde A$ is symmetrizable. 

Let $\tilde G^{\min}$ be the minimal Kac-Moody group of simply connected type associated to $(\tilde I, \tilde A)$ and $\tilde W$ be its Weyl group. Let $\tilde W_{I^\flat}$ and $\tilde W_{I^\sharp}$ be the parabolic subgroup of $\tilde W$ generated by simple reflections in $I^\flat$ and $I^\sharp$ respectively. We have natural identifications $W \to \tilde W_{I^\flat}, w \mapsto w^\flat$ and $W \to \tilde W_{I^\sharp}, w \mapsto w^\sharp$. For $w \in W_K$, $w^\flat=w^\sharp$. For any $w \in W$, we set $(w^\flat)^\natural=w$ and $(w^\sharp)^\natural=w$. 

Similarly, we have natural embedding $G^{\min} \to \tilde L_{I^\flat}, g \mapsto g^\flat$ and $G^{\min} \to \tilde L_{I^\sharp}, g \mapsto g^\sharp$. For $g \in L_K$, $g^\flat=g^\sharp$. 

We denote by $\tilde X$ the flag variety of $\tilde G^{\min}$, and denote by ${}^{I^\flat}  \mathring{\tilde X}^-$, ${}^{I^\flat} \mathring{\tilde X}_-$, ${}^{I^\flat} \mathring{\tilde R}_{-, -}$ the $I^\flat$-Schubert cells, the opposite $I^\flat$-Schubert cells and the open $I^\flat$-Richardson variety respectively.


\subsection{The Kac-Moody group in a Birkhoff-Bruhat atlas} 
Let $r \in W$. The following multiplication maps are  isomorphisms of ind-varieties:
\begin{align}
(  \dot r U  \dot r^{-1} \cap U^{+} )\times  ( \dot r U  \dot r^{-1} \cap U^{-})   &\longrightarrow  \dot r U  \dot r^{-1}, \label{eq:sigma-} \quad
(g_1, g_2)  \mapsto g_1 g_2;  \\
   ( \dot r U  \dot r^{-1} \cap U^{-})    \times ( \dot r U  \dot r^{-1} \cap U^{+}) &\longrightarrow  \dot r U \dot  r^{-1}, \label{eq:sigma+} \quad
   (h_1, h_2)  \mapsto h_1 h_2.
\end{align}

We define morphisms of ind-varieties
\[
\sigma_{r,-}:  \dot r U  \dot r^{-1}  \rightarrow   \dot r U \dot  r^{-1} \cap U^{-}, \qquad g_1 g_2  \mapsto g_2,
\]
and 
\[
 \sigma_{r,+}:  \dot r U \dot  r^{-1}  \rightarrow  \dot  r U  \dot r^{-1} \cap U^{+},  \qquad h_1 h_2  \mapsto h_2.
\]

\begin{lemma}\cite[Proposition~8.2]{GKL} \label{lem:u}
Let $r \in W$. The map 
$$
\s_r=(\s_{r, +}, \s_{r, -}):  \dot r U^-  \dot r \i \to (U^+ \cap \dot  r U^- \dot  r \i) \times (U^- \cap  \dot r U^- \dot  r\i)
$$
 is an isomorphism of ind-varieties. 
\end{lemma}
Note that the isomorphism is compatible with Levi decompositions. The restriction of $\sigma_r$ gives the isomorphism
\[
  \dot r U_{P^-_K}  \dot r \i \to (U^+ \cap  \dot r U_{P^-_K}  \dot  r \i) \times (U^- \cap  \dot r U_{P^-_K}   \dot r\i).
\]
Let $r \in W^K$ and $U_r =   \dot r B^- P^+_K/P^+_K \subset \CP_{K}$.  We have an  isomorphism 
\[
 \dot r U_{P^-_K}  \dot  r \longrightarrow U_r, \qquad g \mapsto g  \dot r P^+_K/P^+_K.
\]

Finally, for $g \in \dot r U^- \dot r^{-1}$, we write $\sigma_{r, \pm}^{\flat / \sharp}(g)$ for $ \sigma_{r, \pm}(g)^{\flat/\sharp}$.

\begin{theorem}\label{thm:BBatlas1}
For $r \in W^K$, we define the map 
\begin{align*}
\tilde{c}_r: U_r &\to \tilde X,\\
g   \dot r P^+_K/P^+_K &\mapsto \s_{r, +}^{\flat} (g) \cdot \dot{ r }^{\flat} ( \dot{ r }^{-1})^\sharp \cdot \s^\sharp _{r, -}(g)^{-1} \cdot  \tilde B^+/\tilde B^+ \text{ for } g \in \dot r U_{P^-_K}  \dot r \i.
\end{align*}
Then $(\tilde{c}_r)_{r \in W^K}$ gives a Birkhoff-Bruhat atlas for $\CP_K$. 
\end{theorem}

\begin{proof}

Since $r \in W^K$, we have 
$
U^+ \cap  \dot r U_{P^-_K}  \dot r = U^+ \cap  \dot r U^- \dot r
$.

Therefore   $\s^{\flat}_{r, +}(U_{P^-_K} ) = (U^+)^\flat \cap  \dot r^\flat (U^-)^\flat ( \dot r \i)^\flat $. On the other hand, we have 
\begin{align*}\dot{ r }^{\flat} ( \dot{ r }^{-1})^\sharp \cdot \sigma^{\sharp}_{r,-}(U_{P^-_K})^{-1} \cdot  \dot{ r }^\sharp  (\dot{ r }^{-1})^{\flat} 
 =  &  \dot r^\flat \Big( ( \dot r^{-1})^{\sharp} (U^-)^{\sharp} \dot  r^{\sharp} \cap (U_{P^-_K} )^{\sharp}  \Big)( \dot r \i)^\flat \\
= &    \dot r^\flat \Big( ( \dot r^{-1})^{\sharp} (U^-)^{\sharp}  \dot r^{\sharp} \cap \tilde{U}^-_{I^\sharp} \cap  \tilde{U}^-_{P^-_{I^\flat}}  \Big)( \dot r \i)^\flat.
\end{align*}
By Proposition~\ref{prop:-LL}, $\tilde{c}_r$ is an  embedding with image 
$
{}^{I^\flat} \mathring{\tilde X}_{\tilde \nu(r, r)} \cap \tilde L_{I^\flat} \tilde L_{I^\sharp} \tilde B^+/\tilde B^+
$.

We then check the stratifications. 
Let $(v, w) \in Q_K$. Suppose that $g \in  \dot r U_{P^-_K}  \dot r \i$ with $g  \dot r P^+_K/P^+_K \in \mathring{\Pi}_{v, w}$. Then there exists $l \in L_J$ such that $g \dot r l \in B^+ \dot w B^+ \cap B^- \dot v B^+$. By \eqref{eq:sigma-} and \eqref{eq:sigma+}, we have 
\[
\s_{r, +}(g) \dot r l \in B^- g \dot r l \subset B^- \dot v B^+ \quad \text{ and } \quad \s_{r, -}(g) \dot r l \in B^+ g \dot r l \subset B^+ \dot w B^+.
\]
Therefore 
\begin{align*} 
 \s^{\flat}_{r, +}(g) \cdot  \dot{ r }^{\flat} ( \dot{ r }^{-1})^\sharp  \cdot \s^\sharp_{r, -}(g)^{-1} 
 = \,\, &(\s_{r, +}(g)  \dot r)^{\flat}  \cdot \bigl((\s_{r, -}(g)  \dot r) \i \bigr)^{\sharp} \\
= \,\,  &(\s_{r, +}(g)  \dot r l)^{\flat} \cdot  \bigl((\s_{r, -}(g)  \dot r l) \i \bigr)^{\sharp} \\ 
 \in  \,\,  & (U^-)^\flat \cdot \dot v^\flat  \cdot (B^{+})^\flat  (B^+)^\sharp  \cdot (\dot w \i)^\sharp  \cdot (B^+)^\sharp \\ 
\subset  \,\,  & (U^-)^\flat \dot v^\flat  \cdot \tilde{U}_{P^+_{I^\flat \cap I^\sharp}}  \tilde{U}^+_{I^\flat \cap I^\sharp} \cdot (\dot w \i)^\sharp (B^+)^\sharp \\ 
 \stackrel{(\heartsuit)}{\subset} \,\,  &{}^{I^\flat} B^+ \dot v^\flat (\dot w \i)^\sharp (B^+)^\sharp={}^{I^\flat} \mathring{\tilde X}^{\tilde \nu(v, w)},
\end{align*}
where $(\heartsuit)$ follows from 
\[
\dot{v}^\flat \tilde{U}_{P^+_{I^\flat\cap I^\sharp}} = \tilde{U}_{P^+_{I^\flat\cap I^\sharp}} \dot{v}^\flat \quad 
\text{ and } \quad 
\tilde{U}^+_{I^\flat \cap I^\sharp} (\dot w \i)^\sharp=(\dot w \i)^\sharp \tilde{U}^+_{I^\flat \cap I^\sharp}. 
\]

By Proposition \ref{prop:LL}, \begin{align*} 
\tilde{c}_r(U_r) &=\cup_{(v, w) \in Q_J} \tilde{c}_r(U_r \cap \mathring{\Pi}_{v, w}) \subset \cup_{(v, w) \in Q_J} {}^{I^\flat}\mathring{\tilde R}_{r^\flat (r \i)^\sharp, v^\flat (w \i)^\sharp} \\ & \subset {}^{I^\flat} \mathring{\tilde X}_{\tilde \nu(r, r)} \cap \tilde L_{I^\flat} \tilde L_{I^\sharp} \tilde B^+/\tilde B^+.
\end{align*}
 
Since $\tilde{c}_r(U_r)={}^{I^\flat} \mathring{\tilde X}_{\tilde \nu(r, r)} \cap \tilde L_{I^\flat} \tilde L_{I^\sharp} \tilde B^+/\tilde B^+$, all the inclusion above are actually equalities. In particular, for any $(v, w) \in Q_J$, $\tilde{c}_r(U_r \cap \mathring{\Pi}_{v, w})={}^{I^\flat}\mathring{\tilde R}_{r^\flat (r \i)^\sharp, v^\flat (w \i)^\sharp}$. The theorem is proved. 
\end{proof}



\section{Combinatorial atlas}\label{sec:EL}
 
In this section, we assume $W$ is an arbitrary Coxeter group and we discuss a combinatorial analog of the geometric ``atlas model'' in Section \ref{sec:BB}.

\subsection{Posets} Let $Q$ be a poset with partial order $\le$. For any $x, y \in Q$, let $[x, y]=\{z \in Q \mid x \le z \le y\}$ be the interval from $x$ to $y$. The covering relation is denote by $\gtrdot$. In other words, $y \gtrdot x$ if $[x, y]=\{x, y\}$. For any $x, y \in Q$ with $x \le y$, a maximal chain from $x$ to $y$ is a finite sequence of elements $y=w_0 \gtrdot w_1 \gtrdot \cdots \gtrdot w_n=x$ for some $n \in \BN$ and $w_0, w_1, \ldots, w_n \in Q$. The number $n$ is called the length of the chain. Note that the maximal chain may not exists in general. 

We say that $Q$ is {\it pure} if for any $x, y \in Q$ with $x \le y$, the maximal chains from $x$ to $y$ always exist and have the same length. Such length is also called the length of the interval $[x, y]$.  A pure poset $Q$ is called {\it thin} if every interval of length $2$ has exactly $4$ elements, i.e. has exactly two elements between $x$ and $y$. 

A subset $C \subset Q$ is called {\it convex} if for any $x, y \in C$, we have $[x,y] \subset C$.

\subsection{EL-Shellability} Now we recall the notion of EL-shellability introduced by Bjorner in \cite{Bj}. 

Suppose that the poset $Q$ is pure. An edge labeling of $Q$ is a map $\l$ from the set of all covering relations in $Q$ to a poset $\L$. The labeling $\l$ sends any maximal chain of an interval of $Q$ to a tuple of $\L$. A maximal chain is called {\it increasing} if the associated tuple of $\L$ is increasing. The edge labeling $\l$ also allows one to order the maximal chains of any interval of $Q$ by ordering the corresponding tuples lexicographically. 

An edge labeling of $Q$ is called {\it EL-shellable} if for  every interval, there exists a unique increasing maximal chain, and all the other maximal chains of this interval are less than this maximal chain (with respect to the lexicographical order). 

For a poset $Q$, we define the augmented poset 
$\hat{Q} = Q \sqcup \{\hat{0}\}$, where $\hat{0}$ is the  minimal element of $\hat{Q}$. 

Thinness and EL-shellability are important combinatorial properties of posets. For example, Bjorner proved in \cite{Bj2} that if a finite poset $\hat{Q}$ is thin and EL-shellable, then it is the face poset of some regular CW complex homeomorphic to a sphere.

The main result of this section is the following. 

\begin{theorem}
The poset $\hat{Q}_K$ is thin and EL-shellable. 
\end{theorem}

The strategy of the proof is as follows: we first establish a combinatorial atlas model for the poset ${Q}_K$; then we prove $\hat{Q}_K$ is thin in \S\ref{sec:thin}; we finally prove $\hat{Q}_K$ is EL-shellable in \S\ref{sec:EL}.

The case where $W$ is a finite Weyl group was first established by Williams in \cite[Theorem 1 \& Theorem 2]{LW}. Another proof for $Q_K$ when $W$ is a finite Weyl group was given in \cite{HL1}. Our  approach to handle $Q_K$ are different. As to the augmented element $\hat{0}$, we follow a similar idea as in \cite{LW}.  


\subsection{Combinatorial atlas}
Let $K \subset I$. Recall $Q_K=\{(v, w) \in W \times W^K \mid v \le w\}$ equipped with the partial order $\preceq$ defined by 
\[
 (v', w') \preceq (v, w)  \text{ if there exists } u \in W_K \text{ such that } v \le v' u \le w' u \le w. 
\]

The construction of $\tilde{W}$ in \S\ref{sec:tildeG} remains valid for arbitrary Coxeter group $W$.

\begin{prop}\label{prop:tW}
	Define the map 
	$$
	\tilde \nu: Q_K  \to \tilde W, (v, w) \mapsto (v )^\flat (w \i)^\sharp. 
	$$ Then
	\begin{enumerate}
	\item the map $\tilde \nu$ induces an isomorphism of the posets $(Q_K, \preceq)$ and $(\tilde \nu(Q_K), \leI)$; 
	    \item $
	\tilde \nu (Q_K)= \{\tilde{w} \in \tW_{I^\flat} \tW_{I^\sharp} \mid  r^\flat (r \i)^\sharp \leI \tilde{w}  \text{ for some } r \in W^K\}$;
	\item $\tilde \nu (Q_K)$ is a convex subset of the poset $(\tW, \leI)$.
	\end{enumerate}
\end{prop}
We shall first give a geometric proof of the proposition in \S\ref{sec:geo} when $W$ is a Weyl group of some Kac-Moody group $G$. We then give a combinatorial proof of the proposition in \S\ref{sec:com} for an arbitrary Coxeter group $W$.


\subsection{Geometric proof of Proposition~\ref{prop:tW}}\label{sec:geo}
In this subsection, we assume that $W$ is a Weyl group of some Kac-Moody group $G$. We deduce Proposition~\ref{prop:tW} for such $W$ as a consequence of Theorem~\ref{thm:BBatlas1}.
 
Recall \S\ref{sec:QK} that the poset $(Q_K, \preceq)$ is the index set of the stratification of $\CP_K$ into the unions of projected Richardson varieties. By \eqref{eq:closure}, the poset $( \tilde{W}, \leI)$ is the index set of the stratification of $\tilde{X}$ into the unions of the $I^\flat$-Schuber varieties. Proposition~\ref{prop:tW} (1) follows directly from  Theorem~\ref{thm:BBatlas1}. 


We show Proposition~\ref{prop:tW} (2). By Proposition~\ref{prop:intersection}, we have that 
\[
c_r(U_r) = \!\! \bigsqcup_{ \stackrel{\tilde{w} \in \tW_{I^\flat} \tW_{I^\sharp},}{ r^\flat (r \i)^\sharp \leI \tilde{w} }} \!\! {}^{I^\flat} \mathring{\tilde X}_{\tilde \nu(r, r)} \cap {}^{I^\flat} \mathring{\tilde X}^{\tilde w}
\] and for any $\tilde{w} \in \tW_{I^\flat} \tW_{I^\sharp}$ with $r^\flat (r \i)^\sharp \leI \tilde{w}$, ${}^{I^\flat} \mathring{\tilde X}_{\tilde \nu(r, r)} \cap {}^{I^\flat} \mathring{\tilde X}^{\tilde w} \neq \emptyset$. Hence $$\tilde{\nu}(\{(v, w) \in Q_K \mid U_r \cap \mathring{\Pi}_{v, w} \neq \emptyset\})=\{\tilde{w} \in \tW_{I^\flat} \tW_{I^\sharp} \mid r^\flat (r \i)^\sharp \leI \tilde{w}\}.$$

We then show Proposition~\ref{prop:tW} (3).    Let $x, y \in \tilde \nu(Q_K)$ with $r^\flat (r \i)^\sharp \leI x$ for some $r \in W^K$.   Since $\tilde{P}^+_{I^\flat}\tilde{P}^+_{I^\sharp}/ \tilde{B}^+$ is closed in $\tilde{G}/\tilde{B}^+$, we see that $\tW_{I^\flat} \tW_{I^\sharp}$  is closed in $\tilde{W}$ under the partial order $\leI$ thanks to Proposition~\ref{prop:LL1}.
Then for any $z \in [x, y]$, since $z \leI y$, we have $z \in \tW_{I^\flat} \tW_{I^\sharp}$. We also have $r^\flat (r \i)^\sharp \leI x \leI z$, hence $z \in \tilde \nu(Q_K)$.

Proposition~\ref{prop:tW} then follows for such $W$.



\subsection{The partial order $\leJ$} 
Let $W$ be an arbitrary Coxeter group and $J$ be a subset of the set of simple reflections in $W$.
In this subsection, we give an equivalent description of the partial order $\leJ$ defined in \S\ref{sec:J-Schubert}.  
\begin{lem}
 Let $x\in W_J$ and $y\in W^J$. The partial order $\leJ$ is generated by 
      \begin{enumerate}
      \item [(a)] $s_\b x y \i {\,\,{}^J\!\! < } x y \i$ for $\b \in \D^{re, -}_J$ with $x < s_\b x$;
      
      \item [(b)] $x s_{\b'} y \i {\,\,{}^J\!\! < } x y \i$ for $\b' \in \D^{re, +}$ with $s_{\b'} y \i<y \i$. 
      \end{enumerate}
\end{lem}
\begin{proof}
 Recall $\Psi_J =\D^{re, -}_J \sqcup (\D^{re, +}-\D^{re, +}_J)$. Let $\b \in \Psi_J$ with $s_{\b} x y \i \leJ x y\i$, hence $y x \i(\b) \in \D^{re, -}$. 
      
      If $\b \in \D^{re, -}_J$, then $x \i(\b) \in \D^{re}_J$. Since $y \in W^J$, $y x \i(\b) \in \D^{re, -}$ is equivalent to $x \i(\b) \in \D^{re, -}_J$. In this case, we equivalently have $s_\b x>x$. 
      
      If $\b \in \D^{re, +}-\D^{re, +}_J$, then $x \i(\b) \in \D^{re, +}-\D^{re, +}_J$. Set $\b'=x \i(\b)$. Then $y(\b') \in \D^{re, -}$. In this case, $s_{\b'} y \i<y \i$. On the other hand, if $s_{\b'} y \i<y \i$ for some $\b' \in \D^{re, +}$, then we must have $\b' \notin \D^{re, +}_J$ since $y \in W^J$. 
            \end{proof}

\begin{cor}\label{cor:leJ1}
Let $x\in W_J$ and $y\in W^J$. The partial order $\leJ$ is generated by 
    \begin{enumerate}
      \item [(a$'$)] $x_1 y \i {\,\,{}^J\!\! < }  x y \i$ for $x_1 \in W_J$ with $x<x_1$. 
      
      \item [(b$'$)] $x y_1 \i {\,\,{}^J\!\! < }  x y \i$ for $y_1 \in W$ with $y_1<y$. 
      \end{enumerate}
\end{cor}

\begin{lemma}\cite[Lemma A.3]{HL} \label{lem:lege}
	Let $x, x', u \in W$ such that $x \le x' u$. Then there exists $u' \le u$, such that $x (u' )\i \le x'$.
%
%
%
%
%
\end{lemma}

\begin{proof}
In the notations of \cite[Lemma A.3]{HL}, Lemma~\ref{lem:lege} is equivalent to 
\[
	x  \vartriangleleft u^{-1}  \le x' \text{ if and only if } x \le x'  \ast u.
\]
The lemma follows.
\end{proof}

\begin{proposition}\label{prop:leJ}
      Let $x, x' \in W_J$ and $y, y' \in W^J$. The following conditions are equivalent: 
      \begin{enumerate}
       \item $x' (y') \i \leJ x y \i$;
       
       \item  There exists $u \in W_J$ such that $x \le x' u$ and $y' u \le y$. 
       \end{enumerate}
\end{proposition}

\begin{proof}
      
      It suffices to replace condition (1) with the generating relations $(a')$ and $(b')$ considered in Corollary~\ref{cor:leJ1}. 
      
      Note that Corollary~\ref{cor:leJ1} $(a')  \Rightarrow(2)$ is trivial. For Corollary~\ref{cor:leJ1} $(b')$, we write $y_1$ as $y' u$ for $y' \in W^J$ and $u \in W_J$. Then $x y_1 \i=(x u \i) (y') \i$. And we have $(x u \i) u=x$ and $y' u \le y$. So $(1)\Rightarrow(2)$ is proved. 
      
      Now we show that $(2)\Rightarrow(1)$. Suppose that there exists $u \in W_J$ such that $x \le x' u$ and $y' u \le y$. By Lemma \ref{lem:lege}, there exists $u' \le u$ such that $x (u') \i \le x'$. Since $y' \in W^J$, $y' u' \le y' u \le y$. By Corollary~\ref{cor:leJ1}, we have 
      \[
      x' (y') \i \leJ x (u') \i (y') \i  =x (y' u') \i \leJ x y \i.  
      \]
      This finishes the proof. 
\end{proof}

\subsection{Combinatorial proof of Proposition~\ref{prop:tW}}\label{sec:com}

We show (1). For $w \in W^K$, we have $w \i \in {}^K W$ and hence $(w \i)^\sharp \in {}^{I^\flat} \tW$. It is easy to see that $\tilde{\nu}$ is injective. We prove the compatibility of the partial orders.  Let 
	\[
	(v, w), (v',  w') \in Q_K.
	\]
By Proposition \ref{prop:leJ}, $(v' )^\flat ((w') \i)^\sharp \,\, \leI (v )^\flat (w \i)^\sharp$  if and only if there exits $z \in \tW_{I^\flat}$ such that $(v )^\flat \le (v' )^\flat z \text{ and } z \i ((w') \i)^\sharp \le  (w \i)^\sharp$.
	
	Note that $((w') \i)^\sharp, (w \i)^\sharp \in \tW_{I^\sharp}$. Thus $z \i ((w') \i)^\sharp \le  (w \i)^\sharp$ implies that $z \in \tW_{I^\sharp} \cap \tW_{I^\flat}$. In this case, $z^\natural \in W_K$. Hence 
	\[
	\begin{split}
	v ^\flat \le (v')^\flat z  
	&\Longleftrightarrow \,\, v  \le v'  z^\natural, \\
	z \i ((w') \i)^\sharp \le  (w \i)^\sharp  
	&\Longleftrightarrow \,\, w' z^\natural \le w.
	\end{split}
	\]
	 So $(v' )^\flat ((w') \i)^\sharp \,\, \leI v ^\flat (w \i)^\sharp$  if and only if $(v',  w') \preceq (v,  w)$. 
	 
	 We show (2). For any $(v,w)\in Q_K$, we have $(w,w) \preceq (v,w)$. Therefore we have $ w^\flat (w^{-1})^\sharp \leI \tilde{\nu}((v,w))$ for $w \in W^K$. Hence $$ \tilde \nu(Q_K) \subset \{\tilde{w} \in \tW_{I^\flat} \tW_{I^\sharp} \mid  r^\flat (r \i)^\sharp \leI \tilde{w}  \text{ for some } r \in W^K\}.$$
	
	 Let $\tilde{w} \in \tW_{I^\flat} \tW_{I^\sharp}$ with $r^\flat (r \i)^\sharp \leI \tilde{w}  \text{ for some } r \in W^K$. We write $\tilde{w} = x y^{-1}$ with $x \in \tW_{I^\flat} $, $y \in  \tW_{I^\sharp} \cap \tW^{I^\flat}$.
	 By Proposition~\ref{prop:leJ}, there exists $u \in W_{I^\flat}$ such that $x \le r^\flat u $ and $r^\sharp u \le y$. Therefore we must have $u \in \tilde{W}_{I^\flat} \cap \tilde{W}_{I^\sharp}$, that is, $u^\natural \in W_K$. So $x^\natural \le r u^\natural \le y^\natural $. Hence we have  $(x^\natural, y^\natural ) \in Q_K$ and $\tilde{\nu} ((x^\natural, y^\natural )) = \tilde{w}$. This shows that $$ \tilde \nu(Q_K) \supset \{\tilde{w} \in \tW_{I^\flat} \tW_{I^\sharp} \mid  r^\flat (r \i)^\sharp \leI \tilde{w}  \text{ for some } r \in W^K\}.$$
	 
	 We show (3). By (2), it suffices to show the set $\tW_{I^\flat} \tW_{I^\sharp}$ is closed in  $(\tilde{W}, \leI)$. Then the argument in the last paragraph of \S\ref{sec:geo} applies.  Let $\tilde{w} \leI x y^{-1}$ with  $x \in \tW_{I^\flat} $, $y \in  \tW_{I^\sharp} \cap \tW^{I^\flat}$. We write $\tilde{w} = x' (y')^{-1}$ with  $x' \in \tW_{I^\flat} $, $y' \in  \tW^{I^\flat}$. By Proposition~\ref{prop:leJ}, there exists $u \in W_{I^\flat}$ such that $x \le x'u$ and $y' u \le y$. Then $y' \le y$ and hence $y' \in \tW_{I^\sharp}$. So $\tilde{w} \in \tW_{I^\flat} \tW_{I^\sharp}$.


\subsection{Thinness} \label{sec:thin}
Dyer proved in \cite[Proposition 2.5]{Dyer} that the poset $(\tilde{W}, \leI)$ is thin. We have shown in Proposition~\ref{prop:tW} that the poset $(Q_K, \preceq)$ can be identified with a convex subset of  the poset $(\tilde{W}, \leI)$. So $Q_K$ is thin. 

As to the rank $2$ intervals involving $\hat{0}$, one can follow the same proof (the last paragraph) as in \cite[Proof of Theorem~1.1]{LW}. We shall omit the details here.

\subsection{Reflection orders}
In order to prove the EL-shellability, we first recall the reflection orders introduced by Dyer in \cite[Definition~2.1]{Dyer1}. Let $T$ be the set of reflections in $\tilde{W}$. A total order $\preceq$ on $T$ is called a {\it reflection order} if for any $s, t \in T$, either $s \prec tst \prec  tstst \prec \cdots$ or $t \prec sts \prec  ststs \prec \cdots$.

For any covering relation $w \gtrdot w'$, we label this edge by the reflection $w (w') \i \in T$. Dyer proved in \cite[Proposition~3.9]{Dyer} (taking $I =\emptyset$ and $J=S$ in {\it loc.cit.}) that any reflection order on $T$ gives an EL-labelling on $(\tW, \leI)$. In particular, the poset $(\tW, \leI)$ is EL-shellable.

Now we take the reflection order $\preceq$ on $T$ such that $\tilde{W}_{I^\flat} \cap T$ is a final section, i.e., for any $t, t' \in T$ with $t \in \tilde{W}_{I^\flat}$ and $t' \notin \tilde{W}_{I^\flat}$, we have $t'\prec t$. The existence of such reflection order was established by Dyer in \cite[Proposition~2.3]{Dyer1}. We define the augmented totally order set $\hat{T} = T \sqcup \{\bot\}$ where $a \prec \bot \prec b$ for any $a \in T - \tilde{W}_{I^\flat}$ and $b \in \tilde{W}_{I^\flat} \cap T$.

Recall we can identify $Q_K$ with a  convex subset of $\tilde{W}$. The edge labelling on $\tilde{W}$ defined above induces an edge labelling on $Q_K$. For any edge of $\hat{Q}_K$ involving $\hat{0}$, we label the edge by $\bot \in \hat{T}$. This gives a labelling on all the edges in $\hat{Q}_K$.

\subsection{EL-shellability}\label{sec:EL}
In this subsection, we show that the edge labelling  on $\hat{Q}_K$ defined above is an EL-labelling. 
We have already seen that the edge labelling above on $(\tilde{W}, \leI)$ is an EL-labelling. 

Let $[x,y]$ be an interval in $\hat{Q}_K$. 

We first consider the case $x \neq \hat{0}$.  By Proposition~\ref{prop:tW}, the poset $Q_K$ can be identified with a convex subset of $(\tilde{W}, \leI)$. Then $\tilde{\nu}([x,y])$ is an interval in $\tilde{W}$. The claim follows in this case.

Now we consider the interval $[\hat{0}, y]$. Similar to the argument in the last two paragraphs of \cite[Proof of Theorem 1.2]{LW}, for any $(v, w) \in Q_K$, any lexicographically minimal chain from $(v, w)$ to $\hat{0}$ does not involve edges labelled by $\tilde{W}_{I^\flat}\cap T$ and must have $(r,r) \succ \hat{0}$ as the last two terms, where $r=\min(v W_J)$. The claim for the interval  $[\hat{0}, (v, w)]$ in $\hat{Q}_K$ now follows from the claim of the interval $[(r,r), (v,w)]$ in $(Q_k, \preceq)$ proved above.


\subsection{Comparison with the work \cite{LW}} 

It is worth pointing out that the edges labelled by $T \cap \tilde{W}_{I^\flat}$ and $\bot$ corresponds to the edges of type 2 and type 3 in the sense of \cite[Corollary 6.4]{LW} respectively. Moreover, we may choose a reflection order on $T$ so that its restriction to $T \cap \tilde{W}_{I^\flat}$ matches with the order used for the type 2 edges in \cite{LW}. However, it is not clear how to match $T\backslash \tilde{W}_{I^\flat}$ with the labeling on the type 1 edges in \cite{LW}. 



\section{A different Birkhoff-Bruhat atlas for $W_K$ finite case}\label{sec:2nd-atlas}
In this section we construction a different Birkhoff-Bruhat atlas for $\CP_K$ when $K$ is of finite type. Let $w_K$ be the longest element of the finite Weyl group $W_K$. We first construct the atlas group $\breve{G}$. We list some examples of such construction in \S\ref{sec:example}.

Let $\breve I$ be the union of two copies of $I$, glued along the $(-w_K)$-graph automorphism of $K$. More precisely, let $I^\flat$ and $I^\sharp$ be the two copies of $I$. Then $\breve I=I^\flat \cup I^\sharp$ with $I^\flat \cap I^\sharp=\{j^\flat; j \in K\}=\{j^\sharp; j \in K\}$. For $j, j' \in K$, $j^\flat=(j')^\sharp$ if and only if $\a_{j'}=-w_K(\a_j)$.  Define ${}^{\natural}: I^\flat \to I, (i^\flat)^{\natural}=i$ and ${}^{\natural'}: I^\sharp \to I, (i^\sharp)^{\natural'}=i$. Then for $\breve i \in I^\flat \cap I^\sharp$, $\breve i^\natural$ and $\breve i^{\natural'}$ are both contained in $K$ and $\a_{\breve i^{\natural'}}=-w_{ K}(\a_{\breve i^{\natural}})$. 

The generalized Cartan matrix $\breve A=(\breve a_{\breve i, \breve i'})_{\breve i, \breve i' \in \breve I}$ is defined as follows: 
\begin{itemize}
	\item for $\breve i, \breve i' \in I^\flat$, $\breve a_{\breve i, \breve i'}=a_{\breve i^\natural, (\breve i')^\natural}$; 
	
	\item for $\breve i, \breve i' \in I^\sharp$, $\breve a_{\breve i, \breve i'}=a_{\breve i^\natural, (\breve i')^\natural}$; 
	
	\item for $\breve i \in \breve I-I^\flat$ and $\breve i' \in \breve I-I^\sharp$, $\breve a_{\breve i, \breve i'}=\breve a_{\breve i', \breve i}=0$. 
\end{itemize}

Note that for $j_1, j'_1, j_2, j'_2 \in K$ with $\a_{j'_1}=-w_K(\a_{j_1})$ and $\a_{j'_2}=-w_K(\a_{j_2})$, we have $a_{j_1, j_2}=a_{j'_1, j'_2}$. Thus for $\breve i, \breve i' \in I^\flat \cap I^\sharp$, we have $a_{\breve i^\natural, (\breve i')^{\natural}}=a_{\breve i^{\natural'}, (\breve i')^{\natural'}}$ and the generalized Cartan matrix $\breve A$ is well-defined.

Let $\breve G$ be the minimal Kac-Moody group of simply connected type  associated to $(\breve I, \breve A)$  and $\breve W$ be its Weyl group. Let $\breve W_{I^\flat}$ and $\breve W_{I^\sharp}$ be the parabolic subgroup of $\breve W$ generated by simple reflections in $I^\flat$ and $I^\sharp$ respectively. We have natural identifications $W \to \breve W_{I^\flat}, w \mapsto w^\flat$ and $W \to \breve W_{I^\sharp}, w \mapsto w^\sharp$. For $w \in W_J$, $w^\flat=w_J w^\sharp w_J \i$. For any $w \in W$, we set $(w^\flat)^\natural=w$ and $(w^\sharp)^{\natural'}=w$. 
 
Similarly to Section~\ref{sec:BB}, we have natural embeddings $G  \to \breve L_{I^\flat}, g \mapsto g^\flat$ and $G  \to \breve L_{I^\sharp}, g \mapsto g^\sharp$.

\begin{lem}\label{lem:springer}
Let $L_{K, \rm{der}}$ be the derived group of $L_K$. Then for any $g \in L_{K, \rm{der}}$, 
\[
(g^\flat)^{-1}=\Psi(\dot w_K   g \dot w_K \i )^\sharp.
\]

\end{lem}

\begin{proof}
Since $g \in L_{K, \rm{der}}$, it suffices to prove that for $i \in K$, 
\[
((x_{\alpha_i}(a))^\flat)^{-1}=\Psi(\dot w_K   x_{\alpha_i}(a) \dot w_K  \i )^\sharp, \quad ((y_{\alpha_i}(a))^\flat)^{-1}=\Psi(\dot w_K  y_{\alpha_i}(a) \dot w_K \i)^\sharp.
\]
We prove the first identity. The second identity is proved in the same way. 

Let $j \in K$ with $w_K (\alpha_i) = -\alpha_j$. Then $(s_j \i w_K )(\alpha_i ) = \alpha_j$. 
It follows from \cite[Proposition~9.3.5]{Springer} and direct computations that
\begin{align*}
\dot w_K x_{\a_i}(a) \dot w_K \i= \dot{s}_j   (\dot{s}_j \i \dot w_K  ) x_{\alpha_i}(a) (\dot{s}_j \i \dot w_K  )^{-1} \dot{s}_j \i = \dot{s}_j  x_{\alpha_j}(a)  \dot{s}_j \i= y_{\alpha_j}(-a).
\end{align*}
The lemma is proved. 
\end{proof}

We denote by $\breve \CB$ the flag variety of $\breve G $, and denote by ${}^{I^\flat} \mathring{\breve X}^-$, ${}^{I^\flat}  \mathring{\breve X}_-$, ${}^{I^\flat}  \mathring{\breve R}_{-, -}$ the $I^\flat$-Schubert cells, the opposite $I^\flat$-Schubert cells and the open $I^\flat$-Richardson varieties in $\breve{\CB}$ respectively. 

\begin{theorem}\label{thm:BBatlas2}
	Suppose that $W_K$ is finite. For any $r \in W^K$, define the map 
\begin{align*}
	\breve{c}_r: U_r &\longrightarrow \breve X, \\
	g  \dot r P^+_K/P^+_K &\mapsto \s^{\flat}_{r, -}(g) \cdot  \dot r^\flat ( \dot w_K)^\flat ( \dot r^{-1})^\sharp \cdot \Psi(\s^\sharp_{r, +}(g)) \cdot \breve B^+/\breve B^+ \text{ for } g \in  \dot r U_{P^-_K}  \dot  r \i.
\end{align*}
Then $(\breve{c}_r)_{r \in W^K}$ gives a Birkhoff-Bruhat atlas for $\CP_K$. 
\end{theorem}

\begin{proof}
We have 
\[
 \s^{\flat}_{r, -}( \dot r U_{P^-_K}   \dot r \i) = (U^-)^\flat \cap  \dot r^\flat (U_{P^-_K})^\flat  ( \dot r \i)^\flat =  (U^-)^\flat \cap ( \dot r \dot w_K)^\flat (U^-)^\flat  ( \dot w_K  \dot r \i)^\flat
\]
and
\begin{align*}
& \dot r^\flat ( \dot w_K)^\flat ( \dot r^{-1})^\sharp \cdot \iota\Big(\s^\sharp_{r, +}( \dot r U_{P^-_K}  \dot r \i)\Big)^{-1}  \cdot   \dot r^\sharp( \dot w_K)^\flat ( \dot r^{-1})^\flat \\
= &  \dot r^\flat ( \dot w_K)^\flat ( \dot r^{-1})^\sharp \cdot \Big( (U^{-})^\sharp \cap ( \dot r U_{P^+_K} \dot  r \i)^\sharp \Big)   \cdot   \dot r^\sharp( \dot w_K)^\flat ( \dot r^{-1})^\flat \\
= &  \dot r^\flat ( \dot w_K)^\flat \cdot    \Big( ( \dot r \i U^{-} \dot  r  )^\sharp \cap ( U_{P^+_K} )^\sharp \Big)     \cdot  ( \dot w_K)^\flat ( \dot r^{-1})^\flat.
\end{align*}
Now it follows from Corollary~\ref{cor:JXw} that $\breve{c}_r$ is an embedding with image ${}^{I^\flat} \!  \mathring{\breve X}^{\breve \nu(r, r)}$.

	Let $(v, w) \in Q_K$. Suppose that $g \in  \dot r U_{P^-_K}  \dot r \i$ with $g  \dot r P^+_K/P^+_K \in \mathring{\Pi}_{v , w}$. Then there exists $l \in L_{K,\rm{der}}$ such that $g \dot r  l \in B^+ \dot w B^+ \cap B^- \dot v   U_{P^+_{K}}$. 
	
	 By \eqref{eq:sigma-} and \eqref{eq:sigma+}, we have 
	\[
	\s_{r, +}(g) \dot r  l   \in B^- g \dot r l    \subset B^- \dot v  U_{P^+_{K}}, \qquad \s_{r, -}(g) \dot r l   \in B^+ g \dot r l   \subset B^+ \dot w B^+.
	\]
	Set $l' =\dot w_K \i l \dot w_K \in L_{K,\rm{der}}$.   Then $ ((l')^\flat)^{-1} = \Psi( \dot{w}_K  l' \dot{w}_K \i)= \Psi(l)^\sharp$ by Lemma~\ref{lem:springer}. 
	
	So 
	\begin{align*}
	&\s^{\flat}_{r, -}(g) \cdot  \dot r^\flat (\dot{w}_K)^\flat ( \dot r^{-1})^\sharp \cdot \Psi(\s^\sharp_{r, +}(g))  \cdot \breve B^+/\breve B^+ \\
	& = (\s^{\flat}_{r, -}(g) \dot  r \dot{w}_K l')^{\flat} \cdot \bigl(\Psi(\s^\sharp_{r, +}(g)  \dot r l   ) \bigr)^{\sharp}  \cdot \breve B^+/\breve B^+ \\
	& =  (\s^{\flat}_{r, -}(g)  \dot r  l   \dot{w}_K )^{\flat} \cdot   \bigl(\Psi(\s^\sharp_{r, +}(g)  \dot r l  ) \bigr)^{\sharp}  \cdot \breve B^+/\breve B^+ \\
 	&\in   (U^+)^\flat \cdot    \dot w^\flat  \cdot  (U^+)^\flat \cdot  ( \dot w_K )^{\flat} \cdot  (U_{P^-_{K}})^\sharp \cdot ( \dot v^{-1})^\sharp \cdot  (U^+)^\sharp \cdot \breve B^+/\breve B^+ \\
	&=  (U^+)^\flat \cdot    \dot w^\flat  \cdot  (U_{P^+_K})^\flat \cdot  ( \dot w_K )^{\flat} \cdot  (U_{P^-_{K}})^\sharp \cdot ( \dot v^{-1})^\sharp \cdot  (U^+)^\sharp \cdot \breve B^+/\breve B^+ \\
	& \subset  (U^+)^\flat \cdot    \dot w^\flat  \cdot  U_{P^-_{I^\flat}} \cdot  ( \dot w_K )^{\flat} \cdot  U_{P^+_{I^\sharp}} \cdot ( \dot v^{-1})^\sharp \cdot  (U^+)^\sharp \cdot \breve B^+/\breve B^+ \\
	&= {}^{I^\flat} B^-  \cdot    \dot w^\flat  \  \cdot  ( \dot w_K )^{\flat}   \cdot ( \dot v^{-1})^\sharp \cdot   \cdot \breve B^+/\breve B^+.
	\end{align*}
	
	By Proposition \ref{prop:LL1}, 
	\begin{align*} 
	\breve{c}_r(U_r) &=\cup_{(v, w) \in Q_K; v \le w} \breve{c}_r(U_r \cap \mathring{\Pi}_{v, w}) \subset \cup_{(v,  w) \in Q_K; v   \le w} {}^{I^\flat}\mathring{\breve R}_{(w w_J)^\flat (v \i)^\sharp, (r w_J)^\flat (r \i)^\sharp} \\ 
	& \subset {}^{I^\flat} \mathring{\breve X}^{\breve \nu(r, r)}.
	\end{align*}
	
	Since $\breve{c}_r(U_r)={}^{I^\flat} \mathring{\breve X}^{\breve \nu(r, r)}$, all the inclusions above are actually equalities. In particular, we have 
	\[
	\breve{c}_r(U_r \cap \mathring{\Pi}_{v, w})={}^{I^\flat}\mathring{\breve R}_{(w w_J)^\flat (v \i)^\sharp, (r w_J)^\flat (r \i)^\sharp}, \quad \text{ for } (v, w) \in Q_K.
	\]
	 The theorem is proved. 
\end{proof}

In the case where $W$ is finite, the multiplication by $\dot w_{I^\flat}$ interchanges the opposite Schubert cells with the $I^\flat$-Schubert cells, where $w_{I^\flat}$ is the longest element of $W_{I^\flat}$. Thus for $W$ finite case, the ``atlas model'' constructed here is essentially the same as the Bruhat atlas constructed in \cite{HKL}. They differ by the multiplication by $\dot w_{I^\flat}$. 



\section{Some examples}\label{sec:example}

\begin{table}[H]
\begin{center}
\scalebox{0.8}{
\begin{tabular}{| c | c | c | c |}
\hline
$\CP_K$
&
$\tilde{G}$
&
$\breve{G}$
&
\text{Bruhat atlas}
\\
\hline
	\begin{tikzpicture}[baseline=0]
	\node  (a) at (0,0) {$\circ$};
	\node (b) at (1,0) {$\circ$};
	\draw (a.east) -- (b);
	\end{tikzpicture}
&
	\begin{tikzpicture}[baseline=0]
	\node (a) at (0,0.5) {\rotatebox{180}{$\triangledown$}};
	\node (b) at (1,0.5) {\rotatebox{180}{$\triangledown$}};
	\node (c) at (0,-0.5) {$\triangledown$};
	\node (d) at (1,-0.5) {$\triangledown$};
	\draw (a) -- (b);
	\draw (c) -- (d);
	\end{tikzpicture}
&
\begin{tikzpicture}[baseline=0]
	\node (a) at (0,0.5) {\rotatebox{180}{$\triangledown$}};
	\node (b) at (1,0.5) {\rotatebox{180}{$\triangledown$}};
	\node (c) at (0,-0.5) {$\triangledown$};
	\node (d) at (1,-0.5) {$\triangledown$};
	\draw (a) -- (b);
	\draw (c) -- (d);
	\end{tikzpicture}
&
\begin{tikzpicture}[baseline=0]
	\node (a) at (0,0.5) {\rotatebox{180}{$\triangledown$}};
	\node (b) at (1,0.5) {\rotatebox{180}{$\triangledown$}};
	\node (c) at (0,-0.5) {$\triangledown$};
	\node (d) at (1,-0.5) {$\triangledown$};
	\draw (a) -- (b);
	\draw (c) -- (d);
	\end{tikzpicture}
\\
\hline
$
\begin{tikzpicture}[baseline=0]
	\node(a) at (0,0) {$\circ$};
	\node (b) at (1,0) {$\bullet$};
	\draw (a) -- (b);
	\end{tikzpicture}
$
&
$
\begin{tikzpicture}[baseline=0]
	\node  (a) at (-1,1) {\rotatebox{180}{$\triangledown$}};
	\node  (b) at (0,0) {$\blacklozenge$};
	\node (c) at (-1,-1) {$\triangledown$};
	\draw (a) -- (b) -- (c);
	\end{tikzpicture}
$
&
$
\begin{tikzpicture}[baseline=0]
	\node  (a) at (-1,1) {\rotatebox{180}{$\triangledown$}};
	\node  (b) at (0,0) {$\blacklozenge$};
	\node (c) at (1,-1) {$\triangledown$};
	\draw (a) -- (b) -- (c);
	\end{tikzpicture}
$
&
$
\begin{tikzpicture}[baseline=0]
	\node  (a) at (-1,1) {\rotatebox{180}{$\triangledown$}};
	\node  (b) at (0,0) {$\blacklozenge$};
	\node (c) at (1,-1) {$\triangledown$};
	\draw (a) -- (b) -- (c);
	\end{tikzpicture}
$
\\
\hline
\begin{tikzpicture}[baseline=0]
	\node (a) at (0,0) {$\circ$};
	\node (b) at (1,0) {$\bullet$};
	\node (c) at (2,0) {$\circ$};
	\draw (a) -- (b) -- (c);
	\end{tikzpicture}

&
\begin{tikzpicture}[baseline=0]
	\node (a) at (-1,1){\rotatebox{180}{$\triangledown$}};
	\node (b) at (1,1) {\rotatebox{180}{$\triangledown$}};
	\node (c) at (0,0) {$\blacklozenge$};
	\node (d) at (-1,-1) {$\triangledown$};
	\node (e) at (1,-1)  {$\triangledown$};
	\draw (a) -- (c) -- (b);
	\draw (e) -- (c) -- (d);
	\end{tikzpicture}

&
\begin{tikzpicture}[baseline=0]
	\node (a) at (-1,1){\rotatebox{180}{$\triangledown$}};
	\node (b) at (1,1) {\rotatebox{180}{$\triangledown$}};
	\node (c) at (0,0) {$\blacklozenge$};
	\node (d) at (-1,-1) {$\triangledown$};
	\node (e) at (1,-1)  {$\triangledown$};
	\draw (a) -- (c) -- (b);
	\draw (e) -- (c) -- (d);
	\end{tikzpicture}
&
\begin{tikzpicture}[baseline=0]
	\node (a) at (-1,1){\rotatebox{180}{$\triangledown$}};
	\node (b) at (1,1) {\rotatebox{180}{$\triangledown$}};
	\node (c) at (0,0) {$\blacklozenge$};
	\node (d) at (-1,-1) {$\triangledown$};
	\node (e) at (1,-1)  {$\triangledown$};
	\draw (a) -- (c) -- (b);
	\draw (e) -- (c) -- (d);
	\end{tikzpicture}
\\
\hline
\begin{tikzpicture}[baseline=0]
	\node (a) at (0,0) {$\bullet$};
	\node (b) at (1,0) {$\circ$};
	\node (c) at (2,0) {$\bullet$};
	\draw (a) -- (b) -- (c);
\end{tikzpicture}
&
\begin{tikzpicture}[baseline=0]
	\node (a) at (0,1){\rotatebox{180}{$\triangledown$}};
	\node (b) at (-1,0) {$\blacklozenge$};
	\node (c) at (1,0) {$\blacklozenge$};
	\node (d) at (0,-1) {$\triangledown$};
	\draw (a) -- (b) -- (d) -- (c) -- (a);
	\end{tikzpicture}
&
\begin{tikzpicture}[baseline=0]
	\node (a) at (0,1){\rotatebox{180}{$\triangledown$}};
	\node (b) at (-1,0) {$\blacklozenge$};
	\node (c) at (1,0) {$\blacklozenge$};
	\node (d) at (0,-1) {$\triangledown$};
	\draw (a) -- (b) -- (d) -- (c) -- (a);
	\end{tikzpicture}
&
\begin{tikzpicture}[baseline=0]
	\node (a) at (0,1) {\rotatebox{180}{$\triangledown$}};
	\node (b) at (-1,0) {$\blacklozenge$};
	\node (c) at (1,0) {$\blacklozenge$};
	\node (d) at (0,-1) {$\triangledown$};
	\draw (a) -- (b) -- (d) -- (c) -- (a);
\end{tikzpicture}
\\
\hline
\begin{tikzpicture}[baseline=0]
	\node (a) at (0,0) {$\circ$};
	\node (b) at (1,0) {$\bullet$};
	\node (c) at (2,0) {$\bullet$};
	\draw (a) -- (b) -- (c);
\end{tikzpicture}
&
\begin{tikzpicture}[baseline=0]
	\node (a) at (0.5,1) {\rotatebox{180}{$\triangledown$}};
	\node (b) at (1,0) {$\blacklozenge$};
	\node (c) at (2,0) {$\blacklozenge$};
	\node (d) at (0.5,-1) {$\triangledown$};
	\draw (a) -- (b) -- (c);
	\draw (b) -- (d);
\end{tikzpicture}
&
\begin{tikzpicture}[baseline=0]
	\node (a) at (0.5,1) {\rotatebox{180}{$\triangledown$}};
	\node (b) at (1,0) {$\blacklozenge$};
	\node (c) at (2,0) {$\blacklozenge$};
	\node (d) at (2.5,-1) {$\triangledown$};
	\draw (a) -- (b) -- (c) -- (d);
\end{tikzpicture}
&
\begin{tikzpicture}[baseline=0]
	\node (a) at (0.5,1) {\rotatebox{180}{$\triangledown$}};
	\node (b) at (1,0) {$\blacklozenge$};
	\node (c) at (2,0) {$\blacklozenge$};
	\node (d) at (2.5,-1) {$\triangledown$};
	\draw (a) -- (b) -- (c) -- (d);
\end{tikzpicture}
\\
\hline
\begin{tikzpicture}[baseline=0]
	\node (a) at (0,0) {$\circ$};
	\node (b) at (1,0) {$\bullet$};
	\node (c) at (2,0) {$\bullet$};
	\draw (a) -- node [above ] {$\scriptstyle\infty$} (b) -- (c);
\end{tikzpicture}
&
\begin{tikzpicture}[baseline=0]
	\node (a) at (0.5,1) {\rotatebox{180}{$\triangledown$}};
	\node (b) at (1,0) {$\blacklozenge$};
	\node (c) at (2,0) {$\blacklozenge$};
	\node (d) at (0.5,-1) {$\triangledown$};
	\draw (a) -- node [  right ] {$\scriptstyle\infty$} (b) -- (c);
	\draw (b) --node [ right] {$\scriptstyle\infty$}   (d);
\end{tikzpicture}
&
\begin{tikzpicture}[baseline=0]
	\node (a) at (0.5,1) {\rotatebox{180}{$\triangledown$}};
	\node (b) at (1,0) {$\blacklozenge$};
	\node (c) at (2,0) {$\blacklozenge$};
	\node (d) at (2.5,-1) {$\triangledown$};
	\draw (a) -- node [ left ] {$\scriptstyle\infty$} (b) -- (c);
	\draw (c) -- node [  left ] {$\scriptstyle\infty$}  (d);
\end{tikzpicture}
&
N/A
\\
\hline
\begin{tikzpicture}[baseline=0]
	\node (a) at (0,0) {$\bullet$};
	\node (b) at (1,0) {$\bullet$};
	\node (c) at (2,0) {$\circ$};
	\draw (a) -- node [above] {$\scriptstyle\infty$} (b) -- (c);

\end{tikzpicture}
&
\begin{tikzpicture}[baseline=0]
	\node (a) at (2.5,1) {\rotatebox{180}{$\triangledown$}};
	\node (b) at (1,0) {$\blacklozenge$};
	\node (c) at (2,0) {$\blacklozenge$};
	\node (d) at (2.5,-1) {$\triangledown$};
	\draw (b) --node [above] {$\scriptstyle\infty$}  (c);
	\draw (c) --  (a);
	\draw (c) -- (d);
\end{tikzpicture}
&
N/A
&
N/A
\\
\hline
\end{tabular}}
\end{center}
\end{table}

In this section, we provide some examples of the Dynkin diagrams of the atlas groups. Here $\tilde{G}$ and $\breve{G}$ are the atlas groups constructed in \S\ref{sec:BB} and \S\ref{sec:2nd-atlas} respectively and ``Bruhat atlas'' refers to the one constructed in \cite{HKL}. Note that even when $\tilde{G}$ and $\breve{G}$ are the same, the atlases constructed  in \S\ref{sec:BB} and \S\ref{sec:2nd-atlas} are totally different.

In the first column, we give the Dynkin diagram for the group $G$. The subset $K$ is the set of vertices filled with black color. In the other columns, the Dynkin diagrams consists of three types of vertices: $\triangle$, the vertices in $(I-J)^\sharp$; $\triangledown$, the vertices in $(I-J)^\flat$; $\blacklozenge$, the vertices in $I^\sharp \cap I^\flat=J^\sharp=J^\flat$. 

We also would like to point out that in certain cases, the atlas groups $\tilde{G}$ and $\breve{G}$ constructed in section \ref{sec:BB} and section \ref{sec:2nd-atlas} coincide. However, the Birkhoff-Bruhat atlases constructed there are still quite different in these cases, as one may see from the maps $\tilde{c}_{r}$ and $\breve{c}_r$ in Theorem~\ref{thm:BBatlas1} and Theorem~\ref{thm:BBatlas2} respectively. 

 

\end{document}